\documentclass[
11pt,twoside, a4paper,
]{amsart}

\title[Groupe fondamental]
{           {\protect\hfill \normalfont \tiny
            \\ \vspace{10pt}}
Le groupe fondamental d'un espace homog\`ene \\
d'un groupe alg\'ebrique lin\'eaire %\\
}

\author{Mikhail Borovoi et Cyril Demarche}

\thanks{M. Borovoi a \'et\'e partiellement soutenu
par le Centre Hermann Minkowski  pour la G\'eom\'etrie \\
\indent C. Demarche a b\'en\'efici\'e d'une aide de l'Agence Nationale de la Recherche portant la r\'ef\'erence ANR-12-BL01-0005.}


\usepackage[all]{xy}
\usepackage{mathptmx}
\usepackage{amscd}
\usepackage{amssymb}
\usepackage{eucal}
\usepackage{amsxtra}
\usepackage{verbatim}
\usepackage{url}
\usepackage{enumerate}
\usepackage[francais]{babel}

\DeclareFontEncoding{OT2}{}{} % to enable usage of cyrillic fonts
\DeclareTextFontCommand{\textcyr}{\fontencoding{OT2}
    \fontfamily{wncyr}\fontseries{m}\fontshape{n}\selectfont}

\usepackage{comment}
\theoremstyle{plain}

\newtheorem{theorem}{Th\'eor\`eme} [section]
\newtheorem{proposition}[theorem]{Proposition}
\newtheorem{lemma}[theorem]{Lemme}
\newtheorem{corollary}[theorem]{Corollaire}

\newtheorem{conditional-result}[theorem]{Conditional Result}

\newtheorem{theorem?}{Theorem(?)} [section]
\newtheorem{proposition?}[theorem]{Proposition(?)}
\newtheorem{lemma?}[theorem]{Lemma(?)}
\newtheorem{corollary?}[theorem]{Corollary(?)}

\newtheorem*{theorem*}{Theorem}
\newtheorem*{proposition*}{Proposition}
\newtheorem*{lemma*}{Lemma}
\newtheorem*{corollary*}{Corollary}
\newtheorem*{question*}{Question}
\newtheorem*{conjecture*}{Conjecture}
\newtheorem*{claim*}{Claim}

\newtheorem*{introtheorem*}{Theorem}
\newtheorem*{introproposition*}{Proposition}
\newtheorem*{introlemma*}{Lemma}
\newtheorem*{introcorollary*}{Corollary}

\theoremstyle{definition}

\newtheorem{notations}[theorem]{Notations}
\newtheorem*{definition*}{Definition}
\newtheorem*{example*}{Example}

\newtheorem{subsec}[theorem]{}

\theoremstyle{remark}
\newtheorem{remark}[theorem]{Remarque}

\newtheorem*{remark*}{Remarque}

\numberwithin{equation}{section}
\DeclareSymbolFont{rsfs}{U}{rsfs}{m}{n}
\DeclareSymbolFontAlphabet{\mathcal}{rsfs}

\DeclareFontEncoding{OT2}{}{} % to enable usage of cyrillic fonts
\DeclareTextFontCommand{\textcyr}{\fontencoding{OT2}
    \fontfamily{wncyr}\fontseries{m}\fontshape{n}\selectfont}

\newcommand{\isoto}{\overset{\sim}{\to}}

\newcommand{\into}{\hookrightarrow}

\newcommand{\labelto}[1]{\xrightarrow{\makebox[1.5em]{\scriptsize ${#1}$}}}
\def\uu{\mathrm{u}}
\def\red{\mathrm{red}}
\def\tor{{\mathrm{tor}}}
\def\sc{{\mathrm{sc}}}
\def\sss{{\mathrm{ss}}}
\def\mult{{\mathrm{mult}}}
\def\ssu{{\mathrm{ssu}}}

\newcommand{\CC}{{\mathbb{C}}}
\newcommand{\ZZ}{{\mathbb{Z}}}
\newcommand{\QQ}{{\mathbb{Q}}}

\newcommand{\Pic}{{\rm Pic}}

\renewcommand{\ker}{{\rm ker}}
\newcommand{\coker}{{\rm coker}}

\newcommand{\Hom}{{\rm Hom}}

\newcommand{\kbar}{{\overline{k}}}

\newcommand{\Gbar}{{\overline G}}
\newcommand{\Hbar}{{\overline H}}
\newcommand{\xbar}{{\overline x}}

\newcommand{\ybar}{{\overline y}}

\newcommand{\That}{{\widehat{T}}}
\newcommand{\Hhat}{{\widehat{H}}}

\newcommand{\Ghat}{{\widehat{G}}}

\newcommand{\Aut}{{\rm Aut}}

\def\G{{\mathbb{G}}}

\def\Hhat{{\widehat{H}}}
\def\Z{{\ZZ}}
\def\Q{{\QQ}}
\def\C{{\CC}}

\def\Ext{{\textup{Ext}}}

\def\top{^{\textup{top}}}
\def\alg{_{\textup{alg}}}

\def\et{{\textup{\'et}}}

\def\top{{\textup{top}}}
\def\an{{}}
\def\alg{{\textup{alg}}}
\def\kerchar{{\textup{\rm kercar}}}
\def\Stab{{\textup{Stab}}}

\def\Zhat{{\widehat{\Z}}}
\def\Gmk{{\G_{m,k}}}

\def\Kbul{{K^\bullet}}
\def\tors{{\rm tors}}
\def\tf{{\rm st}}
\def\AA{{\mathcal{A}}}
\def\sH{{\mathcal{H}}}

\def\tf{{\rm s.t.}}
\def\Zp{\Z_{(p')}}

\def\ii{{\textbf{\textit{i}}}}

\begin{document}

\begin{abstract}
Soit $X$ un espace homog\`ene d'un groupe alg\'ebrique lin\'eaire connexe $G$ sur $\C$.
Soit $x\in X(\C)$. On d\'esigne par $H$ le stabilisateur de $x$ dans $G$.
On montre  que l'on peut d\'efinir alg\'ebriquement le groupe fondamental topologique $\pi_1^\top(X(\C),x)$
si ce groupe fondamental topologique est ab\'elien.
Si $\Pic(G)=0$ et $H$ est connexe ou ab\'elien, on calcule  $\pi_1^\top(X(\C),x)$
en termes des groupes de caract\`eres  de $G$ et $H$.
En outre, si $G$ et $X$ sont d\'efinis sur un corps alg\'ebriquement clos de caract\'eristique $p\ge 0$,
on calcule la partie premi\`ere \`a $p$ du groupe fondamental \'etale de $X$
en termes des groupes de caract\`eres  de $G$ et $H$
(si $\Pic(G)=0$ et $H$ est connexe).

\bigskip

\noindent
{\scshape Abstract.}
Let $X$ be a homogeneous space of a connected linear algebraic group $G$ defined over $\C$.
Let $x\in X(\C)$. We denote by $H$ the stabilizer of $x$ in $G$.
We show that  if the topological fundamental group $\pi_1^\top(X(\C),x)$ is abelian, then it can be defined algebraically.
If $\Pic(G)=0$ and $H$ is connected or abelian, we compute  $\pi_1^\top(X(\C),x)$
in terms of the character groups of $G$ and $H$.
Furthermore, when $G$ and $X$ are defined over an algebraically closed field of characteristic $p\ge 0$,
we compute the prime-to-$p$ \'etale fundamental group of $X$
in terms of the character groups of $G$ and $H$
(if $\Pic(G)=0$ and $H$ is connected).
\end{abstract}

\maketitle

\section{Introduction}

Le groupe fondamental topologique d'un groupe alg\'ebrique lin\'eaire connexe sur  $\C$ a \'et\'e d\'efini alg\'ebriquement
par Merkurjev \cite[\S\,10.1]{Merkurjev} et par le premier auteur   \cite[Def.~1.3]{Bor-Memoir}.
Une troisi\`eme d\'efinition a \'et\'e propos\'ee par Colliot-Th\'el\`ene \cite[Prop.-D\'ef.~6.1]{CT}.
La d\'efinition de  Colliot-Th\'el\`ene a \'et\'e  g\'en\'eralis\'ee par Gonz\'alez-Avil\'es \cite[Def.~3.7]{GA}
et par le premier auteur et Gonz\'alez-Avil\'es  \cite[Def.~2.11]{BG}
aux sch\'emas en groupes r\'eductifs.
Dans cet article  on d\'efinit alg\'ebriquement le groupe fondamental topologique
d'un espace homog\`ene d'un groupe alg\'ebrique lin\'eaire sur $\C$
dans le cas o\`u ce groupe fondamental topologique est ab\'elien,
et on calcule ce groupe fondamental
sous une  certaine condition de connexit\'e  sur le stabilisateur d'un point.
De plus, en utilisant des r\'esultats r\'ecents de Brion et Szamuely \cite{BrSz}
et des r\'esultats classiques sur le groupe fondamental \'etale (voir Szamuely \cite{Sz}),
on consid\`ere le cas o\`u $G$ et $X$ sont d\'efinis sur un corps alg\'ebriquement clos
de caract\'eristique quelconque $p\ge 0$,
et on calcule le groupe fondamental \'etale premier \`a $p$ de $X$
en fonction des groupes de caract\`eres de $G$ et $H$ dans ce cas
(sous la m\^eme hypoth\`ese de connexit\'e pour les stabilisateurs).

\begin{subsec}
Dans tout ce texte, une vari\'et\'e sur un corps $k$ est un $k$-sch\'ema int\`egre, s\'epar\'e et de type fini.
On choisit un \'el\'ement $\ii\in\C$ tel que $\ii^2=-1$.

Soit $X$ une vari\'et\'e  d\'efinie sur $\C$.
Soit $x\in X(\C)$.
On consid\`ere l'espace topologique point\'e $(X(\C),x)$ et le groupe fondamental topologique
$\pi_1^\top(X(\C),x)$.
On \'ecrit $\pi_1^\top(X,x)$ pour $\pi_1^\top(X(\C),x)$.
On pose $\Z(1)=\pi_1^\top(\G_{m,\C}(\C),1)=\pi_1^\top(\C^\times,1)$,
o\`u $\G_{m,\C}$ est le groupe multiplicatif sur $\C$.
On a  un g\'en\'erateur $\xi=\xi_\ii$ (d\'ependent du choix de $\ii$) de $\Z(1)=\pi_1^\top(\C^\times,1)$
donn\'e par le lacet
$$
t\mapsto \exp\,2\pi \ii\, t\colon\quad [0,1]\to \C^\times.
$$
On obtient un isomorphisme $\Z\isoto \Z(1)\colon n\mapsto n\xi$ (d\'ependent du choix de $\ii$).
On pose
$$
\pi_1^\top(X,x)(-1):=\Hom(\Z(1), \pi_1^\top(X,x)),
$$
c'est un ensemble point\'e.
On a une bijection  (d\'ependent du choix de $\ii$)
\begin{equation}\label{eq: phi-mapsto-phi(xi)}
\pi_1^\top(X,x)(-1)=\Hom(\Z(1),\pi_1^\top(X,x))\ \isoto\  \pi_1^\top(X,x),\qquad \phi\mapsto\phi(\xi).
\end{equation}

Si on suppose que le groupe $\pi_1^\top(X,x)$ est ab\'elien,
alors $\pi_1^\top(X,x)(-1)$ est canoniquement un groupe ab\'elien,
et \eqref{eq: phi-mapsto-phi(xi)} est
un isomorphisme de groupes ab\'eliens $\pi_1^\top(X,x)(-1)\isoto \pi_1^\top(X,x)$ (d\'ependent du choix de $\ii$).

Soit $f\colon (\G_{m,\C}, 1)\to (X,x)$ un morphisme de vari\'et\'es point\'ees.
Par fonctorialit\'e on obtient un \'el\'ement
$$
f_*^\top\in \Hom(\Z(1), \pi_1^\top(X,x))=\pi_1^\top(X,x)(-1).
$$
On dit que les \'el\'ements de $\pi_1^\top(X,x)(-1)$ de la  forme $f_*^\top$  sont {\em  alg\'ebriques}.
On d\'esigne par $\pi_1^\top(X,x)(-1)_\alg$
le sous-ensemble point\'e de  $\pi_1^\top(X,x)(-1)$
constitu\'e des \'el\'ements    alg\'ebriques.
\end{subsec}

\begin{subsec}
Soit $k$ un corps alg\'ebriquement clos de caract\'eristique 0.
Soit $X$ une vari\'et\'e  sur $k$, et soit $x\in X(k)$.
On d\'esigne par $\pi_1^\et(X,x)$ le groupe fondamental \'etale de la vari\'et\'e point\'ee $(X,x)$;
voir  Szamuely \cite[Section 5.4]{Sz}.
C'est un groupe topologique.
On pose $\Zhat(1)=\pi_1^\et(\G_{m,k},1)$.

On pose
$$
\pi_1^\et(X,x)(-1):=\Hom_{\textup{cont}}(\Zhat(1), \pi_1^\et(X,x))\, ,
$$
l'ensemble point\'e des homomorphismes continus  de $\Zhat(1)$ vers $\pi_1^\et(X,x)$,
c'est un ensemble point\'e.

Si on suppose que le groupe $\pi_1^\et(X,x)$ est ab\'elien,
alors $\pi_1^\et(X,x)(-1)$ est canoniquement un groupe  ab\'elien.

Soit $f\colon (\G_{m,k}, 1)\to (X,x)$ un morphisme de $k$-vari\'et\'es point\'ees.
Par fonctorialit\'e on obtient un element
$$
f_*^\et\in \Hom_{\textup{cont}}(\Zhat(1), \pi_1^\et(X,x))=\pi_1^\et(X,x)(-1).
$$
On dit que les \'el\'ements de $\pi_1^\et(X,x)(-1)$ de la  forme $f_*^\et$  sont {\em  alg\'ebriques}.
On d\'esigne par $\pi_1^\et(X,x)(-1)_\alg$ le sous-ensemble point\'e de $\pi_1^\et(X,x)(-1)$
constitu\'e des \'el\'ements alg\'ebriques.

\end{subsec}

\begin{subsec}\label{subsec:varkappa}
Soit $(X,x)$ une vari\'et\'e  point\'ee sur $\C$.
Alors $\pi_1^\et(X,x)$ est canoniquement isomorphe  \`a la compl\'etion profinie de $\pi_1^\top(X,x)$
(voir Grothendieck \cite[Expos\'e XII, Corollaire 5.2]{SGA1}),
et en particulier $\Zhat(1)=\pi_1^\et(\G_{m,\C},1)$ est canoniquement isomorphe
\`a la compl\'etion profinie de $\Z(1)=\pi_1^\top(\G_{m,\C},1)$.
On obtient un isomorphisme entre $\Zhat(1)$ et $\Zhat$ (d\'ependent du choix de $\ii$).

Tout homomorphisme $\Z(1)\to\pi_1^\top(X,x)$ induit un homomorphisme continu des compl\'etions $\Zhat(1)\to\pi_1^\et(X,x)$.
Par cons\'equent on obtient une application
$$\varkappa\colon \pi_1^\top(X,x)(-1)\to \pi_1^\et(X,x)(-1).$$
Si le groupe $\pi_1^\top(X,x)$ est ab\'elien et, par cons\'equent, $\pi_1^\et(X,x)$ est ab\'elien,
alors $\varkappa$ est un homomorphisme  des groupes ab\'eliens.
\end{subsec}

\begin{theorem}[Th\'eor\`eme \ref{thm:general-bis}]\label{thm:general}
\begin{enumerate}
\item[(i)]
Soit $X$ un espace homog\`ene d'un groupe alg\'ebrique lin\'eaire connexe $G$ d\'efini sur $\C$.
Alors
 l'application $\varkappa\colon \pi_1^\top(X,x)(-1)\to \pi_1^\et(X,x)(-1)$
induit une bijection
$$
 \pi_1^\top(X,x)(-1)\isoto \pi_1^\et(X,x)(-1)_\alg\subset \pi_1^\et(X,x)(-1).
$$
\item[(ii)]  Si en plus on suppose que le groupe $\pi_1^\et(X,x)$ est ab\'elien,
alors le groupe $\pi_1^\top(X,x)$ est ab\'elien,
le sous-ensemble point\'e $\pi_1^\et(X,x)(-1)_\alg$ du groupe ab\'elien $\pi_1^\et(X,x)(-1)$ est un sous-groupe,
et l'homomorphisme $\varkappa\colon \pi_1^\top(X,x)(-1)\to \pi_1^\et(X,x)(-1)$
induit un isomorphisme de groupes ab\'eliens
$$
 \pi_1^\top(X,x)(-1)\isoto \pi_1^\et(X,x)(-1)_\alg.
$$
\end{enumerate}
\end{theorem}

\def\PP{{\mathbb{P}}}
Comme T.~Szamuely nous l'a fait remarquer, ce r\'esultat ne s'\'etend pas aux groupes alg\'ebriques non lin\'eaires.
En effet, d\'ej\`a pour une vari\'et\'e ab\'elienne $A$ sur $\C$ de dimension positive,
il n'y a pas de morphisme non constant de $\G_{m,\C}$ vers $A$
(car il n'existe pas d'application rationnelle non constante de $\PP^1_\C$ vers $A$, voir \cite[II.1, Cor. du Thm. 4]{Lang}),
alors que $\pi_1^\top(A)\neq 0$.

\begin{corollary}[Corollaire \ref{cor:anti-Serre-bis}]\label{cor:anti-Serre}
Soit $G$ un groupe alg\'ebrique lin\'eaire connexe sur $\C$, et soit $X$ un espace homog\`ene de $G$.
Soit $x\in X(\C)$ et soit $\tau\in\Aut(\C)$.
On suppose que le groupe fondamental \'etale  $\pi_1^\et(X,x)$ est ab\'elien.
Alors  le  groupe fondamental topologique $\pi_1^\top(\tau X,\tau x)$ est canoniquement isomorphe  \`a   $\pi_1^\top(X,x)$.
\end{corollary}

On remarque que ce n'est pas le cas pour des vari\'et\'es lisses quelconques sur $\C$ (avec des groupes fondamentaux \'etales non ab\'eliens),
voir Serre \cite{Serre} et  Milne et Suh \cite{MS}.

%\begin{question}
%Est-il vrai ou non que, pour $X$ comme dans le corollaire \ref{cor:anti-Serre},
%$X(\C)$ et $(\tau X)(\C)$ sont  homotopiquement \'equivalents,
%ou m\^eme hom\'eomorphes, ou m\^eme que $X$ et $\tau X$ sont isomorphes sur $\C$?
%\end{question}

\begin{notations}\label{not}
Soit $k$ un corps alg\'ebriquement clos.
Soit $G$ un groupe alg\'ebrique lin\'eaire connexe d\'efini sur $k$.
On utilise les notations suivantes :

$G^\uu$ est le radical unipotent de $G$;

$G^\red=G/G^\uu$, qui est un groupe r\'eductif;

$G^\sss=[G^\red, G^\red]$, qui est semi-simple;

$G^\sc$ est le rev\^etement universel de $G^\sss$, il est simplement connexe;

$G^\tor=G^\red/G^\sss$, qui est un tore;

$G^\ssu=\ker[G\to G^\tor]$, qui est une extension d'un groupe semi-simple connexe par un groupe unipotent.

On remarque que $G^\tor$ est le plus grand quotient torique de $G$ et que $G^\ssu$ est connexe et sans caract\`eres.

Si $T$ est un tore sur $k$, on \'ecrit $T_*$ pour le groupe des cocaract\`eres de $T$, c'est-\`a-dire
 $T_*:=\Hom_k(\G_{m,k}, T)$.
On a en particulier $T_*\cong\Hom(\That,\Z)$, o\`u $\That$ est le groupe des caract\`eres de $T$.

Soit $H$ un groupe alg\'ebrique lin\'eaire sur $k$.
On \'ecrit $\pi_0(H)=H/H^0$, o\`u $H^0$ est la composante neutre de $H$.
Si $\pi_0(H)$ est ab\'elien, on pose
$$
\pi_0(H)(-1):=\Hom_\Z(\Hom_k(\pi_0(H),\G_{m,k}),\Q/\Z).
$$
On remarque que si $k=\C$, le groupe $\pi_0(H)(-1)$ est isomorphe    \`a $\pi_0(H)$, mais non canoniquement.
\end{notations}

\begin{subsec}
Soit $X$ un espace homog\`ene d'un groupe alg\'ebrique lin\'eaire connexe $G$ d\'efini sur
un corps alg\'ebriquement clos $k$.
On choisit un $k$-point $x\in X(k)$ et on pose $H=\mathrm{Stab}_G(x)$.
On d\'esigne par $H^\mult$ le groupe quotient maximal de $H$ de type multiplicatif.
On pose $H^\kerchar:=\ker [H\to H^\mult]$.
Alors $H^\kerchar$ est l'intersection
des noyaux de tous les caract\`eres $\chi\colon H\to\Gmk$ de $H$.
On suppose :
\begin{enumerate}[\upshape (i)]
    \item  $\Pic(G)=0$,
    \item $H^\kerchar$ est connexe.
\end{enumerate}
On remarque que (i) est satisfait si et seulement si $G^\sss$ est simplement connexe
(voir Sansuc \cite{Sansuc},  Lemme 6.9 et Remarques 6.11.3)
et que (ii) est satisfait  si $H$ est connexe ou  si  $k$ est de caract\'eristique 0 et $H$ est ab\'elien.

On d\'esigne  $\Ghat:=\Hom(G,\Gmk)$ et $\Hhat:=\Hom(H,\Gmk)$.
On \'ecrit
$$
\Ext^0_\Z(\Ghat\to\Hhat,\Z):=\Ext^0_\Z([\Ghat\labelto{i^*}\Hhat\rangle,\Z),
$$
o\`u $[\Ghat\labelto{i^*}\Hhat\rangle$ est un complexe en degr\'es 0 et 1,
et l'homomorphisme $i^*$ est induit par l'inclusion $i\colon H\into G$.
Voir \S\,\ref{subsec:Ext-complex} ci-dessous pour la definition de $\Ext_\Z^0$.
\end{subsec}

\begin{theorem}[Th\'eor\`eme \ref{thm:main-bis}]\label{thm:main}
Soit $X$ un espace homog\`ene d'un groupe alg\'ebrique lin\'eaire connexe $G$
sur $\C$.
Soit $x\in X(\C)$,
on pose $H=\mathrm{Stab}_G(x)$.
On suppose que $\Pic(G)=0$ et que $H^\kerchar$ est connexe.
Alors il existe un isomorphisme canonique de groupes ab\'eliens
$$\pi_1^\top(X,x)(-1)\isoto\Ext^0_\Z(\Ghat\to\Hhat,\Z).$$
\end{theorem}

\begin{corollary}[Corollaire \ref{cor:main-bis}]\label{cor:main}
Sous les hypoth\`eses du th\'eor\`eme \ref{thm:main}

(a) on a une suite exacte canonique
\begin{equation}\label{eq:seq-exacte}
\Hom(\Hhat,\Z)\labelto{i_*}\Hom(\Ghat,\Z)\to\pi_1^\top(X,x)(-1)\to\pi_0(H)(-1)\to 0,
\end{equation}
o\`u $i\colon H\hookrightarrow G$ est l'homomorphisme d'inclusion;

(b) si en plus le sous-groupe $H$ est connexe,
alors la suite exacte \eqref{eq:seq-exacte} induit un isomorphisme canonique
$$
\coker [H^\tor_*\labelto{i_*}G^\tor_*] \isoto \pi^\top_1(X,x)(-1).
$$
\end{corollary}

On remarque que ce corollaire \ref{cor:main}(b) est une version plus explicite de \cite[Thm.~8.5(i)]{BvH2}.

\begin{subsec}
Soit $G,\ X,\ H$ comme dans le th\'eor\`eme \ref{thm:main}, et on suppose que $H$ est connexe.
On choisit des tores maximaux compatibles $T_{H^\sc}\subset H^\sc$ et $T_{H^\red}\subset H^\red$.
On a un homomorphisme canonique $T_{H^\red}\to G^\tor$.
On consid\`ere la cohomologie du complexe de groupes de cocaract\`eres
$$
C_{X *}:=\langle T_{H^\sc *}\to T_{H^\red *}\to G^\tor_*],
$$
o\`u $\langle\, ,\ ]$ signifie que $G^\tor_*$ est en degr\'e $0$.
Ce complexe est isomorphe au dual (au sens du foncteur "Hom interne"
$\mathcal{H}\!\!om_{\Z}^*(\,\cdot\, ,\Z)$\,) du complexe  $\widehat{C}_X$
(ou $\widehat{C}'_X$)  introduit dans \cite{Dem-JOA} et \cite{Dem-Edinb}.

On remarque que $\sH^0(C_{X *})=\coker [H^\tor_*\labelto{i_*}G^\tor_*]$, et donc le corollaire \ref{cor:main}(b) dit que
$\pi_1^\top(X,x)(-1)\cong \sH^0(C_{X *})$, o\`u on \'ecrit $\sH^i(C_{X *})$ pour $i$-\`eme groupe de cohomologie du complexe $C_{X*}$.
\end{subsec}

\begin{theorem}[Th\'eor\`eme \ref{thm:pi2-bis}]\label{thm:pi2}
Soit $X$ un espace homog\`ene d'un groupe alg\'ebrique lin\'eaire connexe $G$
sur $\C$.
Soit $x\in X(\C)$,
on pose $H=\mathrm{Stab}_G(x)$.
On suppose que $\Pic(G)=0$ et que $H$ est connexe.
Alors il existe un isomorphisme canonique de groupes ab\'eliens
$$
\mathcal{H}^{-1}(C_{X*})\isoto \pi^\top_2(X,x)(-1) .$$
\end{theorem}

On remarque que le th\'eor\`eme \ref{thm:pi2} est plus fort que le th\'eor\`eme 8.5(ii) de \cite{BvH2},
o\`u seulement $\pi_2^\top(X,x)(-1)$ modulo torsion a \'et\'e calcul\'e.
Plus explicitement, le th\'eor\`eme \ref{thm:pi2} dit que
$$
\pi_2^\top(X,x)(-1)\cong\ker[T_{H^\red *}\to G^\tor_*]/T_{H^\sc *}\,.
$$

\begin{subsec}
Supposons maintenant que $G$ et $X$ sont d\'efinis sur un corps alg\'ebriquement clos $k$
de caract\'eristique quelconque  $p\ge 0$.
On note $\pi_1^\et(X, x)^{(p')}$ le quotient maximal premier \`a $p$ du groupe fondamental \'etale de $X$.
On d\'efinit
$$
\pi_1^\et(X, x)^{(p')}(-1)=\Hom_{\textup{cont}}(\pi_1^\et(\G_{m,k}, 1)^{(p')},\pi_1^\et(X, x)^{(p')}).
$$
On \'ecrit $\Z_{(p')}$ pour le produit direct des anneaux $\Z_\ell$ pour $\ell\neq p$.

En utilisant  des r\'esultats de Brion et Szamuely \cite{BrSz}, on  d\'emontre le th\'eor\`eme suivant,
qui est une version du corollaire \ref{cor:main}(b) en caract\'eristique positive :
\end{subsec}

\begin{theorem}[Th\'eor\`eme \ref{thm:charp-bis}] \label{thm:charp}
Soit $k$ un corps alg\'ebriquement clos de caract\'eristique $p \geq 0$. Soit $G/k$ un groupe lin\'eaire connexe lisse
et $X/k$ un espace homog\`ene de $G$. Soit $x \in X(k)$, on pose $H := \textup{Stab}_G(x)$.
On suppose que $\Pic(G) = 0$ et que $H$ est lisse et connexe.
 Alors il existe un isomorphisme canonique de groupes topologiques ab\'eliens :
$$\coker[H^\tor_*\labelto{i_*}G^\tor_*] \otimes_{\Z} \Z_{(p')} \isoto \pi_1^\et(X,x)^{(p')}(-1) \, .$$
\end{theorem}

On remarque que le th\'eor\`eme \ref{thm:charp}
g\'en\'eralise le cas particulier
du th\'eor\`eme 1.2(b) de Brion et Szamuely \cite{BrSz} o\`u $G$ est un groupe lin\'eaire,
et a \'et\'e inspir\'e par ce th\'eor\`eme de Brion et Szamuely.
Remarquons \'egalement que l'hypoth\`ese de lissit\'e sur $H$ peut \^etre enlev\'ee (voir \cite{BrSz}, d\'ebut de la section 3).

\smallskip

On peut g\'en\'eraliser le th\'eor\`eme \ref{thm:charp} en assouplissant l'hypoth\`ese de connexit\'e sur le stabilisateur,
comme dans le th\'eor\`eme \ref{thm:main}.

\begin{theorem}[Th\'eor\`eme \ref{thm:charp non connexe-bis}] \label{thm:charp non connexe}
Soit $k$ un corps alg\'ebriquement clos de caract\'eristique $p \geq 0$.
Soit $G/k$ un groupe lin\'eaire connexe lisse et $X/k$ un espace homog\`ene de $G$.
Soit $x \in X(k)$, on pose $H := \textup{Stab}_G(x)$.
On suppose que $\Pic(G) = 0$, $H$ est lisse et $H^\kerchar$ est connexe.
Alors il existe un isomorphisme canonique de groupes topologiques ab\'eliens :
$$\Ext^0_\Z(\Ghat\to\Hhat,\Z) \otimes_{\Z} \Z_{(p')} \isoto \pi_1^\et(X,x)^{(p')}(-1) \, .$$
\end{theorem}

Bien que le th\'eor\`eme \ref{thm:charp} soit un cas particulier du th\'eor\`eme \ref{thm:charp non connexe},
on prouve le th\'eor\`eme \ref{thm:charp} s\'epar\'ement, car il admet une preuve simple via une suite exacte de fibration.
Remarquons \`a nouveau que l'hypoth\`ese de lissit\'e sur $H$ peut \^etre enlev\'ee (voir \cite{BrSz}, d\'ebut de la section 3).

Le plan de l'article est le suivant.
Dans \S\,\ref{s:1}, on prouve  le th\'eor\`eme \ref{thm:general} et le corollaire \ref{cor:anti-Serre}.
Dans \S\,\ref{s:2}, on rappelle les constructions de groupes auxiliaires et d'espaces homog\`enes,
dont nous avons besoin pour notre d\'emonstration des th\'eor\`emes \ref{thm:main} et \ref{thm:charp non connexe}.
Dans \S\,\ref{s:3}, on prouve  le th\'eor\`eme \ref{thm:main},
le corollaire \ref{cor:main} et le th\'eor\`eme \ref{thm:pi2}.
Dans \S\,\ref{s:5}, on prouve le th\'eor\`eme \ref{thm:charp}.
Dans \S\,\ref{s:6}, on prouve le th\'eor\`eme \ref{thm:charp non connexe}.

\section{Espaces homog\`enes  sur $\C$}\label{s:1}

Dans cette section, on d\'emontre le th\'eor\`eme \ref{thm:general} et le corollaire \ref{cor:anti-Serre}.

\begin{proposition}\label{prop:image}
Soit $(X,x)$ une vari\'et\'e point\'ee sur $\C$.
Alors
$$
\varkappa(\pi_1^\top(X,x)(-1)_\alg)=\pi_1^\et(X,x)(-1)_\alg\, ,
$$
o\`u $\varkappa\colon \pi_1^\top(X,x)(-1)\to \pi_1^\et(X,x)(-1)$
est l'application d\'efinie dans \S\,\ref{subsec:varkappa}.
\end{proposition}

\begin{proof}[D\'emonstration]
Soit $f\colon (\G_{m,\C},1)\to (X,x)$ un morphisme de vari\'et\'es point\'ees,
alors $\varkappa(f_*^\top)=f_*^\et$, ce qui d\'emontre la proposition.
\end{proof}

\begin{proposition}\label{prop:str-algebrique}
Soit $X$ un espace homog\`ene d'un groupe alg\'ebrique lin\'eaire connexe $G$ d\'efini sur $\C$,
et $x\in X(\C)$.
Alors tout \'el\'ement de l'ensemble $\pi_1^\top(X,x)(-1)$ est   alg\'ebrique.
\end{proposition}

\begin{proof}[D\'emonstration]
D'abord, on remarque que si $T$ est un $\C$-tore,
alors tout \'el\'ement de $\pi_1^\top(T,1)(-1)=T_*:=\Hom(\G_{m,\C},T)$ est alg\'ebrique.
Si $G$ est un groupe alg\'ebrique lin\'eaire connexe sur $\C$ et $T\subset G$ un tore maximal,
alors par \cite[Prop.~1.11 et Def. 1.3]{Bor-Memoir}
on a un isomorphisme canonique et fonctoriel en $G$\ \  $\pi_1^\alg(G)\to\pi_1^\top(G,1)(-1)$,
o\`u $\pi_1^\alg$ d\'esigne le groupe fondamental alg\'ebrique
d'un groupe  alg\'ebrique lin\'eaire, d\'efini dans \cite{Bor-Memoir}.
Par la d\'efinition de $\pi_1^\alg$, l'homomorphisme $\pi_1^\alg(T)\to\pi_1^\alg(G)$ est surjectif.
Ainsi l'homomorphisme  $\pi_1^\top(T,1)(-1)\to\pi_1^\top(G,1)(-1)$ est surjectif,
donc  tout \'el\'ement de $\pi_1^\top(G,1)(-1)$ est   alg\'ebrique.

Soit $G$ un  groupe lin\'eaire connexe et $X$ un espace homog\`ene de $G$.
Soit $x\in X(\C)$ et soit  $H\subset G$ le stabilisateur de $x$.
On ne suppose pas que $H$ est connexe.
On a un g\'en\'erateur $\xi$ de $\pi_1^\top(\C^\times)$  (d\'ependent du choix de $\ii$), et on peut  oublier la torsion par $(-1)$.
La fibration $G(\C)\to X(\C)$ de fibre $H(\C)$ donne
 une suite exacte
\begin{equation}\label{eq:fibration}
\pi_1^\top(G,1)\to \pi_1^\top(X,x)\to \pi_0(H)\to 1.
\end{equation}
Soit $p_1\in \pi_1^\top(X,x)$.
Montrons que $p_1$ est  alg\'ebrique.

On d\'esigne par $h_0$ l'image de $p_1$ dans $\pi_0(H)$.
Alors $h_0$ est un \'el\'ement semi-simple
et il est donc l'image d'un \'el\'ement semi-simple $h\in H(\C)$.
Par \cite[Thm.~22.2]{Humphreys}, on sait que $h\in T(\C)$ pour un certain tore maximal $T$ de $G$.
On d\'esigne par $M$  le sous-groupe ferm\'e de $H$ engendr\'e par $h$, alors $M\subset T$.
On a un diagramme commutatif exact
$$
\xymatrix{
1\ar[r]   &M\ar[r]\ar[d]     &T\ar[r]\ar[d]    &T/M \ar[r]\ar[d]    &1  \\
1\ar[r]   &H\ar[r]           &G\ar[r]          &G/H  \ar[r]         &1
}
$$
et un diagramme  commutatif exact induit de morphismes de groupes
$$
\xymatrix{
 \pi_1^\top(T,1)\ar[r]\ar[d]  &\pi_1^\top(T/M,1)\ar[r]\ar[d]   &\pi_0(M)\ar[r]\ar[d]  &1  \\
 \pi_1^\top(G,1)\ar[r]        &\pi_1^\top(G/H,x)\ar[r]         &\pi_0(H)\ar[r]        &1.
}
$$

On d\'esigne par $m_0$ l'image de $h\in M(\C)$ dans $\pi_0(M)$,
alors l'image de $m_0$ dans $\pi_0(H)$ est $h_0$.
On voit que  l'image $h_0$ de $p_1$ dans $\pi_0(H)$
est contenue  dans l'image de la fl\`eche verticale de droite.
Nous avons vu que la fl\`eche verticale de gauche est surjective.
Une chasse au diagramme facile montre alors que $p_1$ est contenu dans l'image de la fl\`eche verticale m\'ediane.
Mais $T/M$ est un tore, donc tous les \'el\'ements de $\pi_1^\top(T/M,1)$ sont  alg\'ebriques.
On conclut que  $p_1$ est  alg\'ebrique.
\end{proof}

\begin{corollary}\label{cor:l'image}
Soit $(X,x)$ comme dans la proposition \ref{prop:str-algebrique}.
Alors
$$
\pi_1^\et(X,x)(-1)_\alg=\varkappa(\pi_1^\top(X,x)(-1)).
$$
\end{corollary}

\begin{proof}[D\'emonstration]
 Le corollaire r\'esulte de la proposition \ref{prop:image}
et de la proposition \ref{prop:str-algebrique}.
\end{proof}

\begin{lemma}\label{lemme:injectif}
Soit $X$ un espace homog\`ene d'un groupe alg\'ebrique lin\'eaire connexe $G$ d\'efini sur $\C$.
Alors l'homomorphisme
$$
 \varkappa_0\colon \pi_1^\top(X,x)\to \pi_1^\et(X,x)
$$
est injectif.
\end{lemma}

\begin{proof}[D\'emonstration]
On pose $\Gamma=\pi_1^\top(X,x)$.
Il faut montrer que $\Gamma$ s'injecte dans sa compl\'etion profinie,
i.e. que $\Gamma$ est un groupe r\'esiduellement fini.
On rappelle qu'un groupe $A$ est dit r\'esiduellement fini si
l'intersection des sous-groupes normaux d'indice fini de $A$ est $\{1\}$,
ou, ce qui est \'equivalent, si l'intersection des sous-groupes d'indice fini de $A$ est $\{1\}$.

On d\'esigne par $\Delta$ l'image de $\pi_1^\top(G,1)$ dans $\Gamma=\pi_1^\top(X,x)$
dans la suite exacte \eqref{eq:fibration},
alors on a une suite exacte courte
$$
1\to\Delta\to\Gamma\to\pi_0(H)\to 1.
$$
Par \cite[Prop.~1.11]{Bor-Memoir}
$\pi_1^\top(G,1)$ est un groupe ab\'elien de type fini, donc $\Delta$ est un groupe ab\'elien de type fini,
donc $\Delta$ est r\'esiduellement fini.
Comme  $\Delta$ est un sous-groupe d'indice fini de $\Gamma$,
on conclut que $\Gamma$ est residuellement fini.
\end{proof}

\begin{corollary}\label{cor:injectif}
(i) Soit $X$ un espace homog\`ene d'un groupe alg\'ebrique lin\'eaire connexe $G$ d\'efini sur $\C$.
Alors   l'application
$$
\varkappa\colon \pi_1^\top(X,x)(-1)\to \pi_1^\et(X,x)(-1)
$$
est injective.

(ii) Si en plus on suppose  que le groupe $\pi_1^\et(X,x)$ est ab\'elien,
alors le groupe $\pi_1^\top(X,x)$ est  ab\'elien.
\end{corollary}

\begin{proof}[D\'emonstration]
(i)  On a un diagramme commutatif
$$
\xymatrix{
\pi_1^\top(X,x)(-1) \ar[r]^\varkappa\ar[d]_\sim   &\pi_1^\et(X,x)(-1)\ar[d]^\sim       \\
\pi_1^\top(X,x)     \ar[r]^{\varkappa_0}          &\pi_1^\et(X,x)
}
$$
o\`u les fl\`eches verticales sont applications bijectives.
Par le lemme \ref{lemme:injectif} l'homomorphisme $\varkappa_0$ est injectif,
donc l'application $\varkappa$ est injective.

(ii ) Comme le groupe $\pi_1^\et(X,x)$ est ab\'elien, par le lemme \ref{lemme:injectif} le groupe $\pi_1^\top(X,x)$ est ab\'elien.
\end{proof}

\begin{theorem}\label{thm:general-bis}
\begin{enumerate}
\item[(i)]
Soit $X$ un espace homog\`ene d'un groupe alg\'ebrique lin\'eaire connexe $G$ d\'efini sur $\C$.
Alors
 l'application $\varkappa\colon \pi_1^\top(X,x)(-1)\to \pi_1^\et(X,x)(-1)$
induit une bijection
$$
 \iota\colon \pi_1^\top(X,x)(-1)\isoto \pi_1^\et(X,x)(-1)_\alg\subset \pi_1^\et(X,x)(-1).
$$
\item[(ii)]  Si en plus on suppose que le groupe $\pi_1^\et(X,x)$ est ab\'elien,
alors le groupe $\pi_1^\top(X,x)$ est ab\'elien,
le sous-ensemble point\'e $\pi_1^\et(X,x)(-1)_\alg$ du groupe ab\'elien $\pi_1^\et(X,x)(-1)$ est un sous-groupe,
et l'homomorphisme $\varkappa\colon \pi_1^\top(X,x)(-1)\to \pi_1^\et(X,x)(-1)$
induit un isomorphisme de groupes ab\'eliens
\begin{equation}\label{eq:varkappa-injectif}
\iota\colon \pi_1^\top(X,x)(-1)\isoto \pi_1^\et(X,x)(-1)_\alg.
\end{equation}
\end{enumerate}
\end{theorem}

\begin{proof}
(i) Par le corollaire \ref{cor:injectif}(i)
l'application $\varkappa$ est injective,
et par le corollaire \ref{cor:l'image} son image est $\pi_1^\et(X,x)(-1)_\alg$\,.

(ii) Si en plus le groupe $\pi_1^\et(X,x)$ est ab\'elien, alors par le corollaire \ref{cor:injectif}(ii)
$\pi_1^\top(X,x)$ est aussi ab\'elien,
donc $\pi_1^\top(X,x)(-1)$ et $\pi_1^\et(X,x)(-1)$ sont canoniquement des groupes ab\'eliens,
et $\varkappa$ est un homomorphisme.
Par (i) l'image de l'homomorphism $\varkappa$ est $\pi_1^\et(X,x)(-1)_\alg$,
il est un sous-groupe, et encore par (i) $\varkappa$ induit un isomorphisme de groupes  \eqref{eq:varkappa-injectif}.
\end{proof}

\begin{corollary}\label{cor:anti-Serre-bis}
Soit $G$ un groupe alg\'ebrique lin\'eaire connexe  sur $\C$, et soit $X$ un espace homog\`ene  de $G$.
Soit $x\in X$ et soit $\tau\in\Aut(\C)$.
On suppose que le groupe $\pi_1^\et(X,x)$ est ab\'elien.
Alors  le  groupe fondamental topologique $\pi_1^\top(\tau X,\tau x)$ est canoniquement isomorphe  \`a   $\pi_1^\top(X,x)$.
\end{corollary}

\begin{proof}
On construit un isomorphisme canonique
\[
\pi_1^\top(X^\an,x)(-1) \isoto \pi_1^\top(\tau X^\an,\tau x)(-1)
\]
comme suit:
\[
\xymatrix{
\pi_1^\top(X^\an,x)(-1)\ar@{{}-->}[r]\ar[d]_\iota     &\pi_1^\top(\tau X^\an,\tau x)(-1)\ar[d]^{\tau\iota}\\
\pi_1^\et(X,x)(-1)_\alg \ar[r]^-{\tau_*}         & \pi_1^\et(\tau X,\tau x)(-1)_\alg\, ,
}
\]
 o\`u les fl\`eches verticales $\iota$ et $\tau\iota$ sont des isomorphismes de groupes ab\'eliens.
On choisit une unit\'e imaginaire $\ii\in\C$, alors on obtient
un g\'en\'erateur  $\xi_\ii$ de $\pi_1^\top(\G_{m,\C})$ et
un  isomorphisme compos\'e
$$
\lambda_{\tau,\ii}\colon \pi_1^\top(X^\an,x)\isoto\pi_1^\top(X^\an,x)(-1) \isoto \pi_1^\top(\tau X^\an,\tau x)(-1)\isoto\pi_1^\top(\tau X^\an,\tau x).
$$
Si on change $\ii$ en $-\ii$, alors  $\xi_\ii$ se change en $\xi_{-\ii}=-\xi_\ii$, et l'isomorphisme compos\'e $\lambda_{\tau,\ii}$ ne  se change pas.
Ainsi on obtient un isomorphisme canonique
$$
\lambda_\tau\colon \pi_1^\top(X^\an,x)\isoto \pi_1^\top(\tau X^\an,\tau x).
$$
\end{proof}

\begin{remark}
M\^eme si le groupe $\pi_1^\et(X,x)$ dans le corollaire \ref{cor:anti-Serre-bis} est non ab\'elien, on obtient tout de m\^eme une  bijection canonique d'ensembles point\'es
\begin{equation}\label{bijection-tau}
\pi_1^\top(X^\an,x)(-1)_\alg \isoto \pi_1^\top(\tau X^\an,\tau x)(-1)_\alg\,.
\end{equation}
Les bijections
$$
\pi_1^\top(X,x)\isoto \pi_1^\top(X,x)(-1)\isoto \pi_1^\et(X,x)(-1)_\alg
$$
d\'efinissent une structure de groupe sur l'ensemble point\'e $\pi_1^\et(X,x)(-1)_\alg$
d\'ependant de la topologie et de la structure complexe de $X(\C)$.
De m\^eme on obtient une  structure de groupe sur l'ensemble point\'e $\pi_1^\et(\tau X,\tau x)(-1)_\alg$
d\'ependant de la topologie et de la structure  complexe de $\tau X\,(\C)$.
Mais puisque l'automorphisme $\tau$ de $\C$ ne pr\'eserve pas en g\'en\'eral la topologie et l'unit\'e imaginaire de $\C$,
la bijection \eqref{bijection-tau} n'est pas en g\'en\'eral un isomorphisme de groupes pour ces structures de groupes.
\end{remark}

\def\Gbar{{G}}
\def\kbar{{k}}
\def\Hbar{{H}}
\def\xbar{{x}}
\def\ybar{{y}}
\def\zbar{{z}}
\def\wbar{{w}}

\section{Paires auxiliaires}\label{s:2}

Dans cette section on rappelle les constructions de groupes et d'espaces homog\`enes auxiliaires,
dont nous avons besoin pour notre d\'emonstration des th\'eor\`emes \ref{thm:main} et \ref{thm:charp non connexe}.
L'objectif est d'associer \`a un espace homog\`ene $X$ d'un $k$-groupe alg\'ebrique $G$ v\'erifiant les hypoth\`eses
des th\'eor\`emes \ref{thm:main} et \ref{thm:charp non connexe}, des espaces homog\`enes $Y$, $Z$ et $W$ de certains $k$-groupes
($G_Y$, $G_Z$ et $G_W$ respectivement), avec des morphismes  de paires
$$(G,X) \leftarrow (G_Y, Y) \rightarrow (G_Z, Z) \rightarrow (G_W, W) \, ,$$
qui vont permettre (dans les sections suivantes) de d\'emontrer les th\'eor\`emes \ref{thm:main}
et \ref{thm:charp non connexe} successivement
pour $W$, $Z$, $Y$ et enfin pour $X$.
On utilise pour cela des constructions de \cite{Bor96}, \cite{BCS} et \cite{BS}.

\begin{subsec} {\em Construction de l'espace homog\`ene $Y$.} \label{subsec:aux-Y}
Soit $X$ un espace homog\`ene d'un $k$-groupe alg\'ebrique lin\'eaire connexe lisse $G$
d\'efini sur un corps alg\'ebriquement clos $k$ de caract\'eristique quelconque.
On suppose que $\Pic(\Gbar)=0$.

On choisit un $\kbar$-point $\xbar\in X(\kbar)$. On note $\Hbar$ le stabilisateur de $\xbar$ dans $\Gbar$.
Pour l'\'etude du groupe fondamental de $X$, on peut supposer sans perte de g\'en\'eralit\'e que $\Hbar$ est lisse
(voir \cite{BrSz}, d\'ebut de la section 3).
On ne suppose pas en revanche que $\Hbar$ est connexe.

Soit $\Hbar^\mult$ le plus grand groupe quotient de $\Hbar$ qui est un groupe de type multiplicatif.
On pose $H^\kerchar=\ker[H\to H^\mult]$.
On a un homomorphisme canonique $H^\mult\to G^\tor$, qui n'est g\'en\'eralement pas injectif.

\def\multmap{{m}}

On choisit un plongement $j\colon H^\mult\into Q$
de $H^\mult$ dans un $k$-tore $Q$.
On consid\`ere le plongement
$$
j_*\colon H \to G\times_k Q, \qquad h \mapsto (h,j(\multmap(h))),
$$
o\`u $\multmap\colon H \to H^{\mult}$ est l'\'epimorphisme canonique.
On pose
$$
G_Y=G\times_k Q,\quad H_Y=j_*(H)\subset G_Y,\quad Y=G_Y/H_Y,\quad y=1\cdot H_Y\in Y(k).
$$

La projection $\pi\colon G_Y = G \times Q \to G$ satisfait $\pi(H_Y) = H$, et
elle induit une application $\pi_*\colon Y\to X$ telle que $\pi_*(y)=x$.
On voit  ais\'ement que $Y$ est un torseur sur  $X$ sous le tore $Q$.
On obtient un morphisme  de paires
$$
 (G_Y,Y)\to (G,X).
$$
 On remarque que l'homomorphisme  $H_Y^{\mult}\to G_Y^{\tor}$ est injectif, donc
$$
H_Y\cap G_Y^\ssu=H_Y^\kerchar\cong H^\kerchar.
$$
\end{subsec}

\def\nt{{\rm ssu}}
\begin{subsec} {\em Construction de l'espace homog\`ene $Z$.}
On pose $G_Z=G_Y^\tor=G_Y/G_Y^\nt$, o\`u $G_Y^\nt:=\ker[G_Y\to G_Y^\tor]$.
On a un homomorphisme canonique $\mu\colon G_Y\to G_Z$.
Alors $G_Z$ est un $k$-tore et on a $\widehat{G_Z}=\widehat{G_Y}$.

L'inclusion $i\colon H\into G$ induit un homomorphisme $i^\mult\colon H^\mult\to G^\mult = G^\tor$.
On obtient un plongement
$$
\iota\colon H^\mult\into G^\tor\times_k Q,\quad h\mapsto (i^\mult(h),j(h)).
$$
On pose
$$
Z=Y/G_Y^\nt=(G^\tor\times_k Q)/\iota(H^\mult),
$$
alors on a une application $\mu_*\colon Y\to Z$, dont  la fibre au-dessus du $k$-point
$\zbar:=\mu_*(\ybar)\in Z(\kbar)$ est isomorphe \`a
$$
G_Y^\ssu/(H_Y\cap G_Y^\ssu)\cong G^\ssu / H^\kerchar.
$$
La vari\'et\'e $Z$ est un espace  homog\`ene du tore $G_Z$
de stabilisateur $H_Z=H_Y^\mult\subset G_Y^\tor=G_Z$.
On remarque que
$$
\widehat{H_Z}=\widehat{H_Y^\mult}=\widehat{H_Y}.
$$
Enfin, on a un morphisme  de paires
$$
(G_Y,Y)\to (G_Z,Z).
$$
\end{subsec}

\begin{subsec} {\em Construction de l'espace homog\`ene $W$.}
On pose $G_W=G_Z/H_Z$, $W=Z$, $w=\zbar$,
alors $W$ est un espace principal homog\`ene du tore $G_W$.
On a un morphisme naturel  de paires
$$
(G_Z,Z)\to (G_W, W).
$$
\end{subsec}

\section{Le groupe fondamental topologique}\label{s:3}

Dans cette section on prouve le th\'eor\`eme \ref{thm:main},
le corollaire \ref{cor:main}, et le th\'eor\`eme \ref{thm:pi2}.
On utilise des constructions de \S\,\ref{s:2}.

\begin{subsec}\label{subsec:Ext-complex}
Soit $\Kbul$ un complexe born\'e dans une  cat\'egorie ab\'elienne $\AA$, et soit $B$ un objet de $\AA$.
On d\'efinit
$$
\Ext_\AA^i(\Kbul,B):=\Hom_{D^b(\AA)}(\Kbul,B[i]),
$$
o\`u $D^b(\AA)$ est la   cat\'egorie deriv\'ee des  complexes born\'es dans $\AA$,
et $B[i]$ est le complexe constitu\'e d'un objet $B$ en degr\'e $-i$.
Si $A$ est un objet de $\AA$, on a
$$
\Ext_\AA^i(A[0],B)=\Hom_{D^b(\AA)}(A[0],B[i])=:\Ext_\AA^i(A,B),
$$
voir \cite[Def.~III.5.3]{GM}.
Par d\'efinition $\Ext_\AA^0(A,B)=\Hom_\AA(A,B)$.

On consid\`ere la cat\'egorie des $\Z$-modules (groupes ab\'eliens),
et on \'ecrit $\Ext_\Z^i$ pour $\Ext$ dans cette cat\'egorie.
Soit $A$ un groupe ab\'elien.
On \'ecrit $A_\tors$ pour  le sous-groupe de torsion de $A$,
et on pose $A_\tf:=A/A_\tors$, alors $A_\tf$ est sans torsion.
Il est clair que $\Ext_\Z^0(A,\Z)=\Hom(A,\Z)=\Hom(A_\tf,\Z)$.
\end{subsec}

\begin{lemma}[bien connu]\label{lem:Ext1}
Soit  $A$ un groupe  ab\'elien de type fini, alors $\Ext_\Z^1(A,\Z)\cong\Hom(A_\tors,\Q/\Z)$.
\end{lemma}

\begin{proof}[Id\'ee de la preuve]
En utilisant la r\'esolution injective de $\Z$
$$
0\to\Z\to\Q\to\Q/\Z\to 0,
$$
on montre que
$$
\Ext^1_\Z(A,\Z)\cong\coker[\Hom(A,\Q)\to\Hom(A,\Q/\Z)],
$$
mais
\begin{align*}
\coker[\Hom(A,\Q)\to\Hom(A,\Q/\Z)]=\coker[\Hom(A_\tf,\Q)\to \Hom(A,\Q/\Z)]&\\
                      =\Hom(A,\Q/\Z)/\Hom(A_\tf,\Q/\Z)=\Hom(A_\tors,\Q/\Z)&.
\end{align*}
\end{proof}

\begin{theorem}\label{thm:main-bis}
Soit $X$ un espace homog\`ene d'un groupe alg\'ebrique lin\'eaire connexe $G$
sur $\C$.
Soit $x\in X(\C)$,
on pose $H=\mathrm{Stab}_G(x)$.
On suppose que $\Pic(G)=0$ et que $H^\kerchar$ est connexe.
Alors il existe un isomorphisme canonique de groupes ab\'eliens
$$\pi_1^\top(X,x)(-1)\isoto\Ext^0_\Z(\Ghat\to\Hhat,\Z).$$
\end{theorem}

\begin{subsec}\label{subsec:case-W}
Prouvons le th\'eor\`eme \ref{thm:main-bis}.
On traite d'abord le cas d'un espace homog\`ene principal $W$ d'un $k$-tore $G_W$.
Soit $w\in W(\C)$ un $\C$-point.
L'application $G_W\to W$ d\'efinie par $g\mapsto g\cdot w$ est un isomorphisme de $\C$-vari\'et\'es,
et on a  un isomorphisme induit
$$
\pi_1^\top(G_W,1)(-1)\isoto\pi_1^\top(W,w)(-1).
$$
Comme
$$
\pi_1^\top(G_W,1)(-1)=G_{W*}=\Hom(\widehat{G_W},\Z)=\Ext_\Z^0(\widehat{G_W},\Z),
$$
on obtient un isomorphisme canonique
$$
\pi_1^\top(W,w)(-1)\isoto\Ext^0_\Z(\widehat{G_W},\Z).
$$
Ceci prouve le th\'eor\`eme \ref{thm:main-bis} pour $(G_W,W)$.
\end{subsec}

\begin{subsec}\label{subsec:diagram}
On suppose qu'on a un homomorphisme de $\C$-tores $\gamma_\alpha\colon G_{W'}\to G_W$
et une application $\gamma_\alpha$-\'equivariante d'espaces homog\`enes principaux
$\alpha\colon W'\to W$ envoyant un $\C$-point $w'\in W'(\C)$ sur un $\C$-point $w\in W(\C)$.
Alors le diagramme suivant  commute clairement :
$$
\xymatrix{
\pi_1^\top(W',w')(-1)\ar[d]_\cong \ar[r]^{\alpha_*} &\pi_1^\top(W,w)(-1)\ar[d]^\cong \\
\Ext^0_\Z(\widehat{G_{W'}},\Z)  \ar[r]^{\gamma_{\alpha *}}   &\Ext^0_\Z(\widehat{G_{W}},\Z),
}
$$
o\`u les fl\`eches verticales sont les isomorphismes canoniques de  \S\,\ref{subsec:case-W}.
\end{subsec}

\begin{subsec}\label{subsec:case-Z}
 On a $Z=W$ et $G_Z/H_Z=G_W$, et le morphisme  \'evident de complexes
$\widehat{G_W}\to [\widehat{G_Z}\to\widehat{H_Z}\rangle$
est un  quasi-isomorphisme, donc
$$
\pi_1^\top(Z,z)(-1)=\pi_1^\top(W,w)(-1)=\Ext_\Z^0(\widehat{G_W},\Z)
=\Ext^0_\Z([\widehat{G_Z}\to\widehat{H_Z}\rangle,\Z)
$$
et on obtient un isomorphisme canonique $\pi_1^\top(Z,z)(-1)\isoto\Ext^0_\Z([\widehat{G_Z}\to\widehat{H_Z}\rangle,\Z)$.
Ceci prouve le th\'eor\`eme \ref{thm:main-bis} pour $(G_Z,Z)$.
\end{subsec}

\begin{subsec}\label{subsec:case-Y}
On a une  fibration $G^\ssu(\C)\to G^\ssu(\C)/H^\kerchar(\C)$ de fibre connexe $H^\kerchar(\C)$,
donc  on a une  suite exacte de fibration
$$
1=\pi_1^\top(G^\ssu)\to \pi_1^\top(G^\ssu/H^\kerchar)\to\pi_0(H^\kerchar)=1
$$
(ici $\pi_1^\top(G^\ssu)=1$ parce que $\Pic(G)=0$).
On voit que $\pi_1^\top(G^\ssu/H^\kerchar)=1$.
Or on a une fibration $\mu_*\colon Y(\C) \to Z(\C)$ de fibre
$G^\ssu(\C)/H^\kerchar(\C)$, donc on a une suite exacte de fibration
$$
1=\pi_1^\top(G^\ssu/H^\kerchar)\to\pi_1^\top(Y,y)\labelto{\mu_*}\pi_1^\top(Z,z)\to\pi_0(G^\ssu/H^\kerchar)=1,
$$
Il en r\'esulte que  l'homomorphisme $\pi_1^\top(Y,y)\labelto{\mu_*}\pi_1^\top(Z,z)$ est un isomorphisme,
donc l'homomorphisme $\pi_1^\top(Y,y)(-1)\labelto{\mu_*}\pi_1^\top(Z,z)(-1)$ est un isomorphisme.
Comme $\widehat{G_Y}=\widehat{G_Z}$ et $\widehat{H_Y}=\widehat{H_Z}$, on d\'eduit le th\'eor\`eme \ref{thm:main-bis} pour $(G_Y,Y)$
du th\'eor\`eme \ref{thm:main-bis} pour $(G_Z,Z)$.
\end{subsec}

\def\fone{{ \framebox{\makebox[\totalheight]{1}} }}
\def\ftwo{{ \framebox{\makebox[\totalheight]{2}} }}
\def\fthree{{ \framebox{\makebox[\totalheight]{3}} }}
\def\ffour{{ \framebox{\makebox[\totalheight]{4}} }}

\begin{subsec}\label{subsec:case-X}
On a un torseur $\pi_*\colon Y\to X$  sous le tore $Q$,
d'o\`u on obtient une suite exacte
$$
\pi_1^\top(Q,1)(-1)\labelto{\lambda_*}\pi_1^\top(Y,y)(-1)\labelto{\pi_*}\pi_1^\top(X,x)(-1)\to 0,
$$
o\`u la fl\`eche $\lambda_*$ est induite par l'application
$$
\lambda\colon Q\to Y,\ q\mapsto q\cdot y.
$$
On a une  suite exacte de complexes
$$
0\to (\Ghat\to\Hhat)\to (\widehat{G_Y}\to\Hhat)\to(\widehat{Q}\to 0)\to 0,
$$
d'o\`u on obtient une suite exacte
\begin{equation}\label{eq:exact1}
\Ext^0_\Z(\widehat{Q},\Z)\to \Ext^0_\Z(\widehat{G_Y}\to\Hhat,\Z)\to   \Ext^0_\Z(\Ghat\to\Hhat,\Z)\to  0
\end{equation}
(car d'apr\`es le lemme \ref{lem:Ext1}, $\Ext^1_\Z(\widehat{Q},\Z)=\Hom(\widehat{Q}_\tors,\Q/\Z)=0$).
On obtient un diagramme avec des lignes exactes
\begin{equation}\label{eq:Q}
\xymatrix{
{\pi}_1^\top(Q,1)(-1) \ar[r]^{\lambda_*} \ar[d]^\cong\ar@{}[dr]|{\fone}
     &\pi_1^\top(Y,y)(-1)\ar[r]^{\pi_*}\ar[d]^\cong\ar@{}[dr]|{\ftwo}  &\pi_1^\top(X,x)(-1)\ar[r]\ar@{-->}[d]      &0 \\
\Ext^0_\Z(\widehat{Q},\Z)\ar[r]          &\Ext^0_\Z(\widehat{G_Y}\to\widehat{H_Y},\Z)\ar[r]   &\Ext^0_\Z(\widehat{G}\to\widehat{H},\Z)\ar[r]  &0\, .
}
\end{equation}

On montre que le rectangle $\fone$ commute.
On consid\`ere le diagramme
$$
\xymatrix{
{\pi}_1^\top(Q,1)(-1) \ar[r]^{\lambda_*} \ar[d]^\cong\ar@{}[dr]|{\fone}   &\pi_1^\top(Y,y)(-1)\ar[r]^\cong\ar[d]^\cong\ar@{}[dr]|{\fthree}
                                    &\pi_1^\top(Z,z)(-1)\ar[r]^\cong\ar[d]^\cong \ar@{}[dr]|{\ffour}      &\pi_1^\top(W,w)(-1)\ar[d]^\cong   \\
\Ext^0_\Z(\widehat{Q},\Z)\ar[r]          &\Ext^0_\Z(\widehat{G_Y}\to\widehat{H_Y},\Z)\ar[r]^\cong
&\Ext^0_\Z(\widehat{G_Z}\to\widehat{H_Z},\Z)\ar[r]^\cong  &\Ext^0_\Z(\widehat{G_W},\Z) .
}
$$
Par construction, les rectangles $\fthree$ et $\ffour$ commutent.
D'apr\`es \S\,\ref{subsec:diagram} le grand rectangle $\fone\cup\fthree\cup\ffour$ commute.
Il en r\'esulte que le rectangle $\fone$ commute.

Dans le diagramme exact \eqref{eq:Q}, le rectangle $\fone$ commute,
ce qui permet de d\'efinir la fl\`eche en  pointill\'es faisant commuter le rectangle $\ftwo$.

Ainsi  on obtient un isomorphisme
\begin{equation}\label{eq:iso}
 \pi_1^\top(X,x)(-1)\isoto \Ext^0_\Z(\widehat{G}\to\widehat{H},\Z),
\end{equation}
qui {\em a priori} peut d\'ependre du choix d'un plongement $j\colon H^\mult\into Q$.
\end{subsec}

\begin{subsec} \label{subsec:independant}
Dans  \S\,\ref{subsec:aux-Y}
le torseur $Y\to X$ a \'et\'e construit \`a partir d'un plongement $j\colon H^\mult\into Q$.
Si on choisit un autre plongement $j'\colon H^\mult\into Q'$, on obtient un autre torseur $Y'\to X$ sous $Q'$.
On pose $Q''=Q\times_k Q'$, et on note $j''\colon H^\mult\into Q''$ le plongement diagonal.
On obtient un torseur $Y''\to X$ sous $Q''$ dominant  \`a la fois $Y$ et $Y'$,
et on en d\'eduit facilement que l'isomorphisme \eqref{eq:iso}
ne d\'epend pas du choix du plongement  $j\colon H^\mult\into Q$.
Ceci conclut la preuve du th\'eor\`eme \ref{thm:main-bis}.
\qed
\end{subsec}

\begin{corollary}\label{cor:main-bis}
Sous les hypoth\`eses du th\'eor\`eme \ref{thm:main-bis}

(a) on a une suite exacte canonique
\begin{equation}\label{eq:sequence-exacte}
\Hom(\Hhat,\Z)\labelto{i_*}\Hom(\Ghat,\Z)\to\pi_1^\top(X,x)(-1)\to\pi_0(H)(-1)\to 0,
\end{equation}
o\`u $i\colon H\hookrightarrow G$ est l'homomorphisme d'inclusion;

(b) si en plus le sous-groupe $H$ est connexe,
alors la suite exacte \eqref{eq:sequence-exacte} induit un isomorphisme canonique
$$
\coker [H^\tor_*\labelto{i_*}G^\tor_*] \isoto \pi^\top_1(X,x)(-1).
$$
\end{corollary}

\begin{proof}
La suite exacte courte de complexes
$$
0\to [0\to \Hhat\rangle \to[\Ghat\to\Hhat\rangle\to [\Ghat\to 0\rangle\to 0
$$
induit une suite exacte longue
\begin{equation}\label{eq:exacte-preuve}
\Ext_\Z^0(\Hhat,\Z)\to \Ext_\Z^0(\Ghat,\Z)\to \Ext_\Z^0(\Ghat\to\Hhat, \Z)\to\Ext_\Z^1(\Hhat,\Z)\to\Ext_\Z^1(\Ghat,\Z).
\end{equation}
On a $\Ext_\Z^0(\Hhat,\Z)=\Hom(\Hhat, \Z)$ et $\Ext_\Z^0(\Ghat,\Z)=\Hom(\Ghat,\Z)$.
D'apr\`es le lemme \ref{lem:Ext1}, on a $\Ext_\Z^1(\Ghat,\Z)=\Hom(\Ghat_\tors,\Q/\Z)=0$ et
$$
\Ext_\Z^1(\Hhat,\Z)=\Hom(\Hhat_\tors,\Q/\Z)=\Hom_\Z(\Hom_\C(\pi_0(H),\G_{m,\C}),\Q/\Z)=\pi_0(H)(-1).
$$
D'apr\`es le th\'eor\`eme \ref{thm:main-bis} on peut \'ecrire $\pi_1^\top(X,x)(-1)$
au lieu de $\Ext_\Z^0(\Ghat\to\Hhat, \Z)$ dans \eqref{eq:exacte-preuve}.
Ceci prouve l'assertion (a) du corollaire; l'assertion (b) en r\'esulte imm\'ediatement.
\end{proof}

\begin{theorem}\label{thm:pi2-bis}
Soit $X$ un espace homog\`ene d'un groupe alg\'ebrique lin\'eaire connexe $G$
sur $\C$.
Soit $x\in X(\C)$,
on pose $H=\mathrm{Stab}_G(x)$.
On suppose que $\Pic(G)=0$ et que $H$ est connexe.
Alors il existe un isomorphisme canonique de groupes ab\'eliens
$$
\mathcal{H}^{-1}(C_{X*})\isoto \pi^\top_2(X,x)(-1) .$$
\end{theorem}

\begin{proof}
On \'ecrit la suite exacte de fibration
$$
0=\pi_2^\top(G,1)(-1) \to \pi_2^\top(X,x)(-1) \to \pi_1^\top(H,1)(-1) \to \pi_1^\top(G,1)(-1) \, ,
$$
o\`u l'annulation du groupe $\pi_2^\top(G,1)$ est un th\'eor\`eme d'\'Elie Cartan
(pour le cas de groupes de Lie compacts voir \cite{Borel}).
On \'ecrit \'egalement la suite exacte de cohomologie associ\'ee \`a la suite exacte de complexes
$$
0\to G^\tor_*\to \langle T_{H^\sc *}\to T_{H^\red *}\to G^\tor_*]\to \langle T_{H^\sc *}\to T_{H^\red *}][1]\to 0
$$
(o\`u $ G^\tor_*$ est en degr\'e $0$ dans le complexe central)
et on obtient le diagramme  suivant, \`a lignes exactes
\begin{equation}
\xymatrix@C=5mm{
0\ar[r] &\sH^{-1}(C_{X*})\ar[r]\ar@{-->}[d] &\sH^0(\langle T_{H^\sc *}\to T_{H^\red *}])\ar[r]\ar[d]^\cong & G^\tor_*\ar[d]^\cong \\
0  \ar[r] & \pi_2^\top(X,x)(-1) \ar[r] & \pi_1^\top(H,1)(-1) \ar[r] & \pi_1^\top(G,1)(-1) \,.
}
\end{equation}
L'isomorphisme vertical $\sH^0(\langle T_{H^\sc *}\to T_{H^\red *}])\isoto\pi_1^\top(H,1)(-1)$ dans ce diagramme
a \'et\'e constuit dans \cite[Prop.~1.11]{Bor-Memoir}
(on note que
$$
\sH^0(\langle T_{H^\sc *}\to T_{H^\red *}])=\coker[ T_{H^\sc *}\to T_{H^\red *}] =\pi_1^\alg(H)
$$
est le groupe fondamental alg\'ebrique introduit dans  \cite{Bor-Memoir}).
Comme cet isomorphisme est fonctoriel en $H$, le rectangle \`a droite est commutatif.
Ce diagramme permet finalement de construire l'isomorphisme canonique en pointill\'es
$$
\sH^{-1}(C_{X*})\isoto\pi_2^\top(X,x)(-1) \, ,
$$
ce qui conclut la preuve du th\'eor\`eme.
\end{proof}

\section{Caract\'eristique positive : stabilisateur connexe}\label{s:5}

Dans cette section  on prouve le th\'eor\`eme \ref{thm:charp}.
On commence  par le lemme crucial suivant :

\begin{lemma} \label{lem:sec}
Soient $G,G'$ deux $k$-groupes alg\'ebriques connexes et $f : (G,X,x) \to (G', X', x')$ un morphisme d'espaces homog\`enes
\`a stabilisateurs respectifs $H := \textup{Stab}_G(x)$ et $H' := \textup{Stab}_{G'}(x')$,
de sorte que le morphisme $G \to G'$ soit surjectif.
On note $G_0 := \ker(G \to G')$ et $X_0 := f^{-1}(x')$. On suppose que $H'$, $H$  et $G_0$ sont connexes.
Alors on a une suite exacte de groupes :
$$
\pi_1^\et(X_0,x)^{(p')} \to \pi_1^\et(X,x)^{(p')} \xrightarrow{f_*} \pi_1^\et(X',x')^{(p')} \to 1 \, .
$$
\end{lemma}

\begin{proof}
On d\'efinit les $k$-groupes lin\'eaires $G_1 := G \times_{G'} H'$ et $H_1 := H \times_{G} G_1$. Alors on a une suite exacte canonique
$$
1 \to G_0 \to G_1 \to H' \to 1 \, ,
$$
donc en particulier le groupe $G_1$ est connexe.
V\'erifions maintenant que $X_0$ est naturellement un espace homog\`ene de $G_1$, de stabilisateur $H_1$ :
en restreignant l'action de $G_1$ sur $X$ \`a la sous-vari\'et\'e $X_0$, on obtient un morphisme $m : G_1 \times X_0 \to X$.
V\'erifions que $X_0$ est stable par cette action de $G_1$, c'est-\`a-dire que le morphisme $m$ se factorise en un morphisme $G_1 \times X_0 \to X_0$.
L'image de $G_1$ par le morphisme $G\to G'$ est le sous-groupe $H'$ de $G'$, et l'image de $X_0$ par le morphisme $f$ est le point $x'$,
dont le stabilisateur dans $G'$ est exactement $H'$.
Par cons\'equent, on voit que le morphisme compos\'e $G_1 \times X_0 \xrightarrow{m} X \xrightarrow{f} X'$
est le morphisme constant \'egal \`a $x'$ (i.e. il se factorise par $x' : \textup{Spec}(k) \to X'$),
ce qui assure que le morphisme $m$ se factorise par l'inclusion $X_0 = f^{-1}(x') \to X$, et donc $m$ d\'efinit bien une action de $G_1$ sur $X_0$.
Il est alors clair que cette action est transitive et que le stabilisateur de $x$ pour cette action est exactement le sous-groupe $H_1 = H \times_G G_1$ de $G_1$.
\smallskip

     Montrons la surjectivit\'e de $f_*$. En utilisant \cite{SGA1}, expos\'e IX, corollaire 5.6,
il suffit de v\'erifier que $f$ est un morphisme universellement submersif \`a fibres g\'eom\'etriquement connexes,
ce qui r\'esulte du fait que $f$ est fid\`element plat et quasi-compact
(voir \cite{SGA3}, expos\'e $\textup{VI}_{\textup{B}}$, proposition 9.2.(xiii).a) :
les deux morphismes $G \to G'$ et $G' \to X'$ sont fid\`element plats et de pr\'esentation finie),
 ainsi que du fait que $G_1$ est connexe.
\smallskip

\item Montrons maintenant l'exactitude de la suite en $\pi_1^\et(X)^{(p')}$.
Pour cela, on utilise le th\'eor\`eme 1.2 de \cite{BrSz}.
En effet, suivant \cite{Sz}, corollaire 5.5.9, il suffit de montrer que pour tout rev\^etement \'etale galoisien $Y \to X$,
de degr\'e premier \`a $p$, tel que $Y \times_{X} X_0$ ait une section sur $X_0$,
il existe un rev\^etement \'etale fini connexe $Y' \to X'$, de degr\'e premier \`a $p$,
tel qu'une composante connexe de $Y' \times_{X'} X$
soit munie d'un $X$-morphisme surjectif vers $Y$. Soit donc un rev\^etement \'etale galoisien $Y \to X$ (de degr\'e premier \`a $p$)
tel que $Y_0 := Y \times_{X} X_0$ admette une section $s_0 : X_0 \to Y_0$.
Comme $H$ est connexe,
le th\'eor\`eme 1.2 de \cite{BrSz} assure qu'il existe un groupe lin\'eaire connexe $\widetilde{G}$,
une isog\'enie centrale $\widetilde{G} \to G$ et un relev\'e $\widetilde{H}$ de $H$ dans $\widetilde{G}$
tel que $Y = \widetilde{G} / \widetilde{H}$. On a donc un diagramme commutatif exact de la forme :
\begin{displaymath}
    \xymatrix{
    1 \ar[r] & \mu \ar[r] & \widetilde{G} \ar[r] \ar@{->>}[d] & G \ar[r] \ar@{->>}[d] & 1 \\
         &            & Y = \widetilde{G} / \widetilde{H} \ar[r]  & X = G/H \, ,}
\end{displaymath}
o\`u $\mu$ est un $k$-groupe fini de type multiplicatif.

On d\'efinit alors le $k$-groupe $\widetilde{G_1} := \widetilde{G} \times_G G_1$.

Consid\'erons le diagramme commutatif suivant :
\begin{displaymath}
\xymatrix{
& \widetilde{G} \ar[rr] \ar'[d][dd] & & G \ar[dd] \\
\widetilde{G_1} \ar[rr] \ar[ru] & & G_1 \ar[ru] \ar[dd] & \\
& Y \ar'[r][rr] & & X \\
Y_0 \ar[rr] \ar[ru] & & X_0 \ar[ru] & \, ,
}
\end{displaymath}
o\`u les quatre carr\'es sont cart\'esiens (pour la face de droite, c'est une cons\'equence de la d\'efinition de $G_1$).
Ce diagramme implique l'existence d'une fl\`eche verticale $\widetilde{G_1} \to Y_0$ qui fait commuter le cube. Puisque
$$\widetilde{G_1} = G_1 \times_G \widetilde{G} = (X_0 \times_X G) \times_G \widetilde{G} = X_0 \times_X \widetilde{G} = X_0 \times_X (Y \times_X G)$$
et
$$Y_0 \times_{X_0} G_1 = Y_0 \times_{X_0} (X_0 \times_X G) = Y_0 \times_X G = (X_0 \times_X Y) \times_X G \, ,$$
on v\'erifie que dans le cube pr\'ec\'edent, tous les carr\'es sont cart\'esiens, donc en particulier le carr\'e suivant
\begin{displaymath}
    \xymatrix{
    \widetilde{G_1} \ar[r]^\pi \ar[d] & G_1 \ar[d] \\
    Y_0 \ar[r] & X_0
    }
\end{displaymath}
est cart\'esien. On voit aussi que le morphisme $(\widetilde{G_1}, Y_0) \to (G_1,X_0)$ est un morphisme d'espaces homog\`enes.

La section $s_0 : X_0 \to Y_0$ induit alors une section $s_1 : G_1 \to \widetilde{G_1}$ du morphisme $\pi$
(comme morphisme de $k$-vari\'et\'es) apparaissant dans le carr\'e pr\'ec\'edent. Par construction, on a $s_1(1) = 1$.
Or $G_1$ est connexe, donc le lemme de Rosenlicht assure que $s_1$ est un homomorphisme de groupes alg\'ebriques
(voir par exemple la preuve de la proposition 3.2 de \cite{CT}).

Remarquons \'egalement que l'on dispose d'un diagramme commutatif de suites exactes courtes centrales
(o\`u la suite exacte sup\'erieure est obtenue en tirant en arri\`ere la suite exacte inf\'erieure par le morphisme injectif $G_1 \to G$) :
\begin{displaymath}
\xymatrix{
1 \ar[r] & \mu \ar[r] \ar[d]^= & \widetilde{G_1} \ar[r]_\pi \ar[d] & G_1 \ar@/_/[l]_{s_1} \ar[r] \ar[d] & 1 \\
1 \ar[r] & \mu \ar[r] & \widetilde{G} \ar[r] & G \ar[r] & 1 \, .
}
\end{displaymath}

La section $s_1$ permet d'identifier $\widetilde{G_1}$ avec le produit direct $G_1 \times \mu$, et donc $G_1$ avec la composante neutre de $\widetilde{G_1}$.
En particulier, via $s_1$, $G_1$ est un sous-groupe distingu\'e de $\widetilde{G}$.
Avec ces identifications, $\widetilde{H}$ est un sous-groupe de $G_1$ dans le groupe $\widetilde{G}$.

On d\'efinit alors $Y' := \widetilde{G} / G_1$.

Si on note maintenant $q$ l'isog\'enie initiale $q : \widetilde{G} \to G$, on a un diagramme commutatif de suites exactes courtes
\begin{displaymath}
\xymatrix{
1 \ar[r] & \widetilde{H} \ar[r] \ar[d] & q^{-1}(H) \ar[d] \ar[r] & \mu \ar[r] \ar[d]^= & 1 \\
1 \ar[r] & G_1 \ar[r] & \widetilde{G_1} \ar[r] & \mu \ar[r] & 1 \, ,
}
\end{displaymath}
qui assure que le carr\'e suivant
\begin{displaymath}
\xymatrix{
Y = \widetilde{G}/\widetilde{H} \ar[r] \ar[d] & X = \widetilde{G}/q^{-1}(H) \ar[d] \\
Y' = \widetilde{G}/G_1 \ar[r] & X' = \widetilde{G}/\widetilde{G_1}
}
\end{displaymath}
est cart\'esien, et que les deux morphismes horizontaux sont des torseurs connexes sous le groupe fini de type multiplicatif $\mu$.

 En particulier, le morphisme $Y' \to X'$ est un rev\^etement \'etale fini connexe, tel que $Y' \times_{X'} X \cong Y$ au-dessus de $X$.
Cela conclut la preuve de l'exactitude de la suite du lemme.
\end{proof}

\begin{theorem}\label{thm:charp-bis}
Soit $k$ un corps alg\'ebriquement clos de caract\'eristique $p \geq 0$. Soit $G/k$ un groupe lin\'eaire connexe lisse
et $X/k$ un espace homog\`ene de $G$. Soit $x \in X(k)$, on pose $H := \textup{Stab}_G(x)$.
On suppose que $\Pic(G) = 0$ et que $H$ est lisse et connexe.
 Alors il existe un isomorphisme canonique de groupes ab\'eliens :
$$\coker[H^\tor_*\labelto{i_*}G^\tor_*] \otimes_{\Z} \Z_{(p')} \isoto \pi_1^\et(X,x)^{(p')}(-1) \, .$$
\end{theorem}

D\'emontrons le th\'eoreme \ref{thm:charp-bis} :
on va traiter d'abord le cas des tores, puis des groupes lin\'eaires connexes, et enfin celui des espaces homog\`enes.

\begin{lemma}[bien connu] \label{lem:tore}
Soit $T$ un tore d\'efini sur $k$.
Alors il y a un isomorphisme canonique et fonctoriel $T_*\otimes_\Z \Zp \isoto \pi_1^\et(T)^{(p')}(-1)$.
\end{lemma}
\begin{proof}
Soit $Y\to T$ un rev\^etement \'etale galoisien de degr\'e $n$ premier \`a $p$.
Par \cite{Miy} ou \cite[Prop.~1.1(a)]{BrSz}, le rev\^etement $Y\to T$
a une structure d'isog\'enie centrale $T'\to T$.
Il en r\'esulte que $Y\to T$ est domin\'e par l'isog\'enie $\varphi_n\colon T\to T,\ t\mapsto t^n$.
On pose $T_n=\ker\,\varphi_n$ (consid\'er\'e comme un groupe abstrait), alors $T_n=\mu_n\otimes_\Z T_*$.
On a :
\begin{align*}
\pi_1^\et(T)^{(p')}&(-1)=\Hom_{\textup{cont}}(\Zp(1),\pi_1^\et(T)^{(p')})=\varprojlim \Hom_{\textup{cont}}(\Zp(1),T_n)\\
      &=\varprojlim\Hom_{\textup{cont}}(\mu_n,\mu_n\otimes_\Z T_*)=\varprojlim\Z/n\Z\otimes_\Z T_*
     =\Zp\otimes_\Z T_*.
\end{align*}
\end{proof}

\begin{lemma}[bien connu] \label{lem:uu-sc}
Soit $G$ :

(a) un groupe unipotent connexe sur $k$, ou

(b) un groupe semi-simple simplement connexe sur $k$.

\noindent Alors $\pi_1^\et(G)^{(p')}=1$.
\end{lemma}

\begin{proof}
Par \cite{Miy} ou \cite[Prop.~1.1(a)]{BrSz}, tout rev\^etement  \'etale galoisien $Y\to G$  de degr\'e $n$ premier \`a $p$
admet une structure d'une isog\'enie centrale $G'\to G$,
mais $G$ comme dans (a) ou (b) n'admet pas d'isog\'enie centrale non triviale de degr\'e premier \`a $p$.
\end{proof}

On consid\`ere maintenant le cas d'un groupe lin\'eaire connexe lisse quelconque.
Si $G$ est un tel groupe, on note $G^\uu$ son radical unipotent et $G^\red := G / G^\uu$.
Soit $T_G$ un tore maximal de $G^\red$ et $T_{G^\sc}$ un tore maximal de $G^\sc$ dont l'image dans $G^\red$ est contenue dans $T_G$.
On d\'efinit alors (voir \cite{Bor-Memoir})
$$
\pi_1^\alg(G) := \coker[T_{G^\sc *} \to T_{G*}] = T_{G*} / T_{G^\sc*} \, .
$$

\begin{remark}\label{rq pi1}
Dans le cas particulier o\`u $\Pic(G) = 0$ (ce qui \'equivaut au fait que $G^\sss$ soit simplement connexe),
la formule pr\'ec\'edente se simplifie en $\pi_1^\alg(G) = {G^\tor}_*$.
\end{remark}

\begin{proposition} \label{prop:pi1alggroup}
Soit $G$ un $k$-groupe lin\'eaire connexe lisse. Alors on a un isomorphisme canonique
$$
\pi_1^\alg(G) \otimes_{\Z} \Z_{(p')} \isoto \pi_1^\et(G)^{(p')}(-1) \, .
$$
\end{proposition}

\begin{proof}
Tout d'abord, on se ram\`ene au cas o\`u $G$ est r\'eductif : en effet,
le morphisme $G \to G^\red$ satisfait les hypoth\`eses du lemme \ref{lem:sec},
donc on en d\'eduit une suite exacte
$$
\pi_1^\et(G^\uu)^{(p')} \to \pi_1^\et(G)^{(p')} \to \pi_1^\et(G^\red)^{(p')} \to 0 \, .
$$
Or $\pi_1^\et(G^\uu)^{(p')} = 0$ d'apr\`es le lemme \ref{lem:uu-sc}(a),
donc on a un isomorphisme $\pi_1^\et(G)^{(p')} \isoto \pi_1^\et(G^\red)^{(p')}$
et on peut donc supposer $G$ r\'eductif.

Dans ce cas, il existe une r\'esolution  de $G$, not\'ee
\begin{equation} \label{resfl}
1 \to S \to G' \to G \to 1 \, ,
\end{equation}
o\`u $S$ est un $k$-tore central et $G'$ est un $k$-groupe  reductif tel que ${G'}^\sss$ est simplement connexe.

Montrons la proposition pour $G'$.
On dispose d'une suite exacte courte
$$
1 \to {G'}^\sss \to G' \to {G'}^\tor \to 1 \, ,
$$
o\`u ${G'}^\sss$ est semi-simple simplement connexe.
On consid\`ere
le diagramme commutatif
\begin{equation} \label{diag Gred}
\xymatrix@C=5mm{
0 \ar[r]  & \coker[T_{G'^\sss *}\to T_{G' *}] \otimes \Z_{(p')} \ar[r] \ar@{-->}[d]^\cong
        & {G'}^\tor_* \otimes \Z_{(p')} \ar[d]^{\cong} \ar[r] & 0 \\
 {\pi}_1^\et({G'}^\sss)^{(p')}(-1) \ar[r] & \pi_1^\et(G')^{(p')}(-1)\ar[r] & \pi_1^\et({G'}^\tor)^{(p')}(-1)\ar[r] & 0 \, .
}
\end{equation}
Par le lemme \ref{lem:sec}, la deuxi\`eme ligne  du diagramme est exacte,
et on a $\pi_1^\et({G'}^\sss)^{(p')}(-1)=0$ d'apr\`es le lemme \ref{lem:uu-sc}(b), car ${G'}^\sss$ est simplement connexe.
De la suite exacte courte
$$
1\to T_{{G'}^\sss}\to T_{G'}\to {G'}^\tor\to 1
$$
on obtient une suite exacte courte
$$
0\to T_{{G'}^\sss*}\to T_{G'*}\to {G'}^\tor_*\to 0,
$$
d'o\`u des isomorphismes
$$
\coker[T_{{G'}^\sss *}\to T_{G'*}]\isoto {G'}^\tor_* \quad \text{et} \quad
\coker[T_{{G'}^\sss *}\to T_{G' *}]\otimes \Zp\isoto {G'}^\tor_*\otimes \Zp
$$
et l'exactitude de la premi\`ere ligne du diagramme \eqref{diag Gred}.
Ce diagramme induit un isomorphisme canonique en pointill\'es,
ce qui d\'emontre la proposition pour $G'$.

D\'eduisons-en le r\'esultat pour $G$.
On applique le lemme \ref{lem:sec} \`a la suite exacte courte \eqref{resfl},
et on obtient le diagramme commutatif suivant, dont la seconde ligne est exacte :
\begin{displaymath}
\xymatrix@C=7mm{
S_* \otimes \Z_{(p')} \ar[r] \ar[d]^{\cong} & \pi_1^\alg(G') \otimes \Z_{(p')} \ar[d]^{\cong} \ar[r]
          & \pi_1^\alg(G) \otimes \Z_{(p')} \ar@{-->}[d] \ar[r] & 0 \\
{\pi}_1^\et(S,1)^{(p')}(-1) \ar[r] &\pi_1^\et(G',1)^{(p')}(-1)\ar[r] & \pi_1^\et(G,1)^{(p')}(-1)\ar[r] & 0 \, .
}
\end{displaymath}
De la suite exacte \eqref{resfl} on obtient une suite exacte courte
$$
0\to S_*\to\pi_1^\alg(G')\to\pi_1^\alg(G)\to 0,
$$
d'o\`u, en vertu de l'exactitude \`a droite du produit tensoriel, l'exactitude de la premi\`ere ligne du diagramme.
Ce diagramme assure finalement l'existence de l'isomorphisme canonique en pointill\'es
et conclut la preuve de la proposition \ref{prop:pi1alggroup}.
\end{proof}

\begin{corollary} [{cf.~\cite[Lemme 3 et bas de la page 152]{Miy}, voir aussi \cite[Prop.~1.1(b)]{BrSz}}]
Le groupe $\pi_1^\et(G)^{(p')}(-1)$ est un quotient de $T_{G*}\otimes_\Z \Zp$, o\`u $T_G$ est un tore maximal de $G$.
\end{corollary}

\begin{subsec} {\em D\'emonstration du th\'eor\`eme \ref{thm:charp-bis}.}
On remarque que l'on a toujours un morphisme canonique $\coker[\pi_1^\alg(H) \to \pi_1^\alg(G)] \to \coker[H^\tor_* \to G^\tor_*]$.
On applique alors le lemme \ref{lem:sec} au morphisme $(G,G) \to (G,X)$, et on obtient le diagramme commutatif suivant
dont la seconde ligne est exacte :
\begin{equation} \label{diag pi_1}
\xymatrix@C=3mm{
\pi_1^\alg(H) \otimes \Z_{(p')} \ar[r] \ar[d]^{\cong} & \pi_1^\alg(G) \otimes \Z_{(p')} \ar[d]^{\cong} \ar[r]
        & \coker[H^\tor_* \to G^\tor_*] \otimes \Z_{(p')} \ar@{-->}[d] \ar[r] & 0 \\
{\pi}_1^\et(H,1)^{(p')}(-1) \ar[r] &\pi_1^\et(G,1)^{(p')}(-1)\ar[r] & \pi_1^\et(X,x)^{(p')}(-1)\ar[r] & 0 \, .
}
\end{equation}
De la suite exacte courte
$$
1\to H^\sss\to H^\red\to H^\tor\to 1
$$
on d\'eduit une suite exacte courte
$$
0\to \pi_1^\alg(H^\sss)\to \pi_1^\alg(H)\to H^\tor_*\to 0,
$$
voir \cite[Lemme 3.7]{BKG} et \cite[Prop. 6.8]{CT}
(dans \cite{BKG} et \cite{CT}, il est suppos\'e que le corps $k$ est de caract\'eristique nulle,
mais la suite est exacte pour $k$ de caract\'eristique quelconque,
voir \cite[Thm. 3.14]{GA} et \cite[Thm. 3.8]{BG}).
Comme $\pi_1^\alg(H^\sss)$ est un groupe fini et $ H^\tor_*$ est un groupe sans torsion, on voit que
$\pi_1^\alg(H)_\tf=H^\tor_*$
(voir \S \ref{subsec:Ext-complex} pour la notation $_\tf$).
D'autre part, comme $G^\sss$ est simplement connexe, on a
$\pi_1^\alg(G)=G^\tor_*$ (voir remarque \ref{rq pi1}).
On obtient que
$$
\coker[\pi_1^\alg(H)\to\pi_1^\alg(G)]=\coker[H^\tor_*\to G^\tor_*]
$$
et que
$$
\coker[\pi_1^\alg(H)\otimes\Zp\to\pi_1^\alg(G)\otimes\Zp]=\coker[H^\tor_*\to G^\tor_*]\otimes\Zp \, ,
$$
donc la premi\`ere ligne du diagramme \eqref{diag pi_1} est exacte.
Finalement, ce diagramme permet bien de d\'efinir l'isomorphisme souhait\'e (en pointill\'es)
$$
\coker[H^\tor_* \to G^\tor_*] \isoto \pi_1^\et(X,x)^{(p')}(-1) \, .
$$
\qed
\end{subsec}

\section{Caract\'eristique positive : stabilisateur non connexe} \label{s:6}

\begin{theorem}\label{thm:charp non connexe-bis}
Soit $k$ un corps alg\'ebriquement clos de caract\'eristique $p \geq 0$.
Soit $G/k$ un groupe lin\'eaire connexe lisse et $X/k$ un espace homog\`ene de $G$.
Soit $x \in X(k)$, on pose $H := \textup{Stab}_G(x)$.
On suppose que $\Pic(G) = 0$, $H$ est lisse et $H^\kerchar$ est connexe.
Alors il existe un isomorphisme canonique de groupes ab\'eliens :
$$\Ext^0_\Z(\Ghat\to\Hhat,\Z) \otimes_{\Z} \Z_{(p')} \isoto \pi_1^\et(X,x)^{(p')}(-1) \, .$$
\end{theorem}

Pour prouver le th\'eor\`eme \ref{thm:charp non connexe-bis},
on commence par \'etendre le th\'eor\`eme 1.2(a) de \cite{BrSz} en supprimant l'hypoth\`ese de connexit\'e sur le stabilisateur :

\begin{proposition} \label{prop:BrSzvar}
Soit $G$ un $k$-groupe  connexe lisse, $X$ un $k$-espace homog\`ene de $G$, $x \in X(k)$.
Soit $\widetilde{X} \to X$ un rev\^etement \'etale galoisien de degr\'e premier \`a $p$.
Alors  il existe une isog\'enie centrale $\pi : \widetilde{G} \to G$ et un sous-groupe d'indice fini $\widetilde{H}$ de $\pi^{-1}(H)$
tel que le morphisme naturel $\widetilde{G} / \widetilde{H} \to X$
se factorise en un isomorphisme $\widetilde{G} / \widetilde{H} \isoto \widetilde{X}$.
\end{proposition}

\def\Gtil{{\widetilde{G}}}
\def\Htil{{\widetilde{H}}}
\begin{proof}
On fixe un point $\widetilde{x} \in \widetilde{X}(k)$ au-dessus de $x$.
On consid\`ere le morphisme quotient $\varphi : G \to X$ d\'efini par l'action de $G$ sur le point $x$,
et on consid\`ere le diagramme cart\'esien suivant :
\begin{displaymath}
\xymatrix{
Y \ar[d]^{\pi} \ar[r]^{\widetilde{\varphi}} & \widetilde{X} \ar[d] \\
G \ar[r]^{\varphi} & X \, ,
}
\end{displaymath}
o\`u $Y=G \times_{X} \widetilde{X}$.
On note $F$ la fibre de $Y$ au-dessus de $1\in G(k)$, et on note $\Gamma=\Aut(Y/G)$,
alors $\Gamma$ agit transitivement sur $F$.
La vari\'et\'e  $Y$ n'est pas connexe en g\'en\'eral.
On note $\widetilde{G}$ la composante connexe de $Y$
contenant le point marqu\'e  $y=(1, \widetilde{x})$.
Alors la restriction $\pi_\Gtil$ de $\pi$ \`a $\widetilde{G}$ est un rev\^etement \'etale de $G$.

Prouvons que $\pi_\Gtil\colon \Gtil\to G$ est un rev\^etement galoisien.
On d\'efinit $F_\Gtil=F\cap\Gtil$ et
$$
\Gamma_\Gtil=\{\gamma\in\Gamma\ |\ \gamma(y)\in F_\Gtil\}\;.
$$
Alors $\Gamma_\Gtil=\Stab_\Gamma(\Gtil)$,
donc $\Gamma_\Gtil$ est un sous-groupe de $\Gamma$,
$\Gamma_\Gtil$ agit sur $\Gtil$ au-dessus $G$,
et $\Gamma_\Gtil$ agit transitivement sur $F_\Gtil$.
On voit que $\Aut(\Gtil/G)$ agit transitivement sur $F_\Gtil$\,,
donc $\Gtil\to G$ est un rev\^etement galoisien.

Alors la proposition 1.1(a) de \cite{BrSz}
assure que la vari\'et\'e $\Gtil$ a une structure de groupe alg\'ebrique sur $k$,
telle que
$\pi_\Gtil : \widetilde{G} \to G$ soit une isog\'enie centrale de $k$-groupes.
On peut supposer que $1_\Gtil=y:=(1_G, \widetilde{x})$.
Or on dispose du diagramme commutatif suivant :
 \begin{displaymath}
\xymatrix{
\widetilde{G} \ar[d]^{\pi_\Gtil} \ar[r]^{\widetilde{\varphi}_\Gtil} & \widetilde{X} \ar[d] \\
G \ar[r]^{\varphi} & X \, .
}
\end{displaymath}
On d\'efinit $\widetilde{H}:= \widetilde{\varphi}_\Gtil^{-1}(\widetilde{x})$.

Montrons que $\widetilde{H}$ est un sous-groupe alg\'ebrique de $\widetilde{G}$.
On fixe $h\in \Htil(k)$.
Pour $g=1_\Gtil\in \Gtil(k)$ on a
$$
\widetilde{\varphi}_\Gtil(gh)=\widetilde{\varphi}_\Gtil(h)=\widetilde{x}=\widetilde{\varphi}_\Gtil(g).
$$
Comme $\Gtil$ est connexe, le corollaire 5.3.3 de \cite{Sz} assure que $\widetilde{\varphi}_\Gtil(gh)=\widetilde{\varphi}_\Gtil(g)$
pour tout $g\in G(k)$.
On voit que si $h\in \Htil(k)$, le morphisme $\widetilde{\varphi}_\Gtil$ est $h$-invariant \`a droite.
Inversement, si le morphisme $\widetilde{\varphi}_\Gtil$ est $h$-invariant \`a droite,
alors $\widetilde{\varphi}_\Gtil(h)=\widetilde{x}$ et donc $h\in\Htil(k)$.
Ainsi $\Htil\subset\Gtil$ est le stabilisateur de $\widetilde{\varphi}_\Gtil$ et c'est donc un sous-groupe algebrique.

Comme le morphisme  $\widetilde{\varphi}_\Gtil$ est $\Htil$-invariant, il induit un morphisme naturel
$\overline{\varphi}_\Gtil : \widetilde{G} / \widetilde{H} \to \widetilde{X}$ qui est un rev\^etement \'etale fini connexe.
La d\'efinition de $\widetilde{H}$ assure que la fibre de $\overline{\varphi}_\Gtil$
au-dessus de $\widetilde{x}$ consiste en un seul point,
donc $\overline{\varphi}_\Gtil$ est un isomorphisme de vari\'et\'es.
\end{proof}

On montre ensuite la variante suivante du lemme \ref{lem:sec} :

\begin{lemma} \label{lem:sec2}
Soient $G,G'$ deux $k$-groupes alg\'ebriques connexes et $f : (G,X,x) \to (G', X', x')$ un morphisme d'espaces homog\`enes
\`a stabilisateurs respectifs $H := \textup{Stab}_G(x)$ et $H' := \textup{Stab}_{G'}(x')$,
de sorte que les morphismes $G \to G'$ et $H \to H'$ soient surjectifs. On note $G_0 := \ker(G \to G')$ et $X_0 := f^{-1}(x')$.
On suppose $G_0$ connexe. Alors on a une suite exacte de groupes :
$$
\pi_1^\et(X_0,x)^{(p')} \to \pi_1^\et(X,x)^{(p')} \xrightarrow{f_*} \pi_1^\et(X',x')^{(p')} \to 1 \, .
$$
\end{lemma}

\begin{proof}
On v\'erifie que $X_0$ est naturellement un espace homog\`ene de $G_0$, de stabilisateur $H_0 := \ker(H \to H')$.
\smallskip

     Montrons d'abord la surjectivit\'e de $f_*$. En utilisant \cite{SGA1}, expos\'e IX, corollaire 5.6,
il suffit de v\'erifier que $f$ est un morphisme universellement submersif \`a fibres g\'eom\'etriquement connexes,
ce qui r\'esulte du fait que $f$ est fid\`element plat et quasi-compact, ainsi que du fait que $G_0$ est connexe.
\smallskip

 Montrons maintenant l'exactitude de la suite en $\pi_1^\et(X)^{(p')}$.
Pour cela, on utilise la proposition \ref{prop:BrSzvar} et des arguments similaires \`a ceux de la preuve du lemme \ref{lem:sec}.
Soit un rev\^etement \'etale galoisien $Y \to X$ (de degr\'e premier \`a $p$) tel que $Y_0 := Y \times_{X} X_0$
admette une section $s_0 : X_0 \to Y_0$. La proposition \ref{prop:BrSzvar} assure qu'il existe un groupe lin\'eaire connexe $\widetilde{G}$,
une isog\'enie centrale $\pi : \widetilde{G} \to G$ et un sous-groupe $\widetilde{H}$ de $\widetilde{G}$ d'indice fini dans $\pi^{-1}(H)$
tel que le morphisme naturel $\widetilde{G} / \widetilde{H} \to X$ se factorise
en un isomorphisme $\widetilde{G} / \widetilde{H} \to Y$.
On d\'efinit $Y_0 := Y \times_X X_0$ et le $k$-groupe $\widetilde{G_0} := \widetilde{G} \times_G G_0$.
Puisque $G_0$ est connexe, la section $s_0 : X_0 \to Y_0$ induit une section $\widetilde{s_0} : G_0 \to \widetilde{G_0}$,
dont on v\'erifie que c'est un morphisme de groupes. Cela permet d'identifier $G_0$ avec la composante neutre de $\widetilde{G_0}$,
et donc de voir $G_0$ comme un sous-groupe distingu\'e de $\widetilde{G}$. On note alors $\widetilde{G'} := \widetilde{G} / G_0$.
D\'efinissons $Y'$ comme le quotient de $\widetilde{G'}$ par l'image $\widetilde{H} / \widetilde{H_0}$ de $\widetilde{H}$ dans $\widetilde{G'}$.
Alors $Y' \to X'$ est un rev\^etement \'etale fini connexe, et par construction on a $Y' \times_{X'} X \cong \widetilde{G} / \widetilde{H}$
au-dessus de $X$, donc on a une factorisation $Y' \times_{X'} X \cong \widetilde{G} / \widetilde{H} \isoto Y \to X$.
Cela conclut la preuve de l'exactitude de la suite du lemme.
\end{proof}

\begin{subsec} {\em D\'emonstration du th\'eor\`eme \ref{thm:charp non connexe-bis}.}
Si $X$ est un espace homog\`ene satisfaisant les hypoth\`eses du th\'eor\`eme \ref{thm:charp non connexe-bis},
on reprend les constructions auxiliaires de la section \ref{s:2}. On dispose des morphismes surjectifs de paires
$$
(G,X) \leftarrow (G_Y,Y) \to (G_Z, Z) \to (G_W, W) \, .
$$
On v\'erifie facilement que chacun des deux premiers morphismes de paires v\'erifie les hypoth\`eses du lemme \ref{lem:sec2}.
Par cons\'equent, le lemme \ref{lem:sec2} assure que les suites naturelles suivantes
$$
\pi_1^\et(G^\sss/H^\kerchar,y)^{(p')} \to \pi_1^\et(Y,y)^{(p')} \to \pi_1^\et(Z,z)^{(p')} \to 1
$$
$$
\pi_1^\et(Q,1)^{(p')} \to \pi_1^\et(Y,y)^{(p')} \to \pi_1^\et(X,x)^{(p')} \to 1
$$
sont exactes.
Puisque le morphisme $G^\sss \to G^\sss / H^\kerchar$ est universellement submersif et \`a fibres g\'eom\'etriquement connexes,
le corollaire 5.6 de \cite{SGA1}, expos\'e IX, assure que le morphisme
$\pi_1^\et(G^\sss,1)^{(p')} \to \pi_1^\et(G^\sss/H^\kerchar,y)^{(p')}$
est surjectif. Or $\pi_1^\et(G^\sss,1)^{(p')} = 0$ car $G^\sss$ est semi-simple simplement connexe.
Donc finalement $\pi_1^\et(G^\sss/H^\kerchar,y)^{(p')} = 0$, et
on a  un isomorphisme canonique
\begin{equation}\label{eq:YZ}
\pi_1^\et(Y,y)^{(p')} \isoto \pi_1^\et(Z,z)^{(p')} \, .
\end{equation}

On va maintenant d\'emontrer le th\'eor\`eme pour $(G_W,W)$, puis pour $(G_Z,Z)$, puis pour $(G_Y,Y)$, et enfin pour $(G,X)$.

Comme $W$ est un espace principal homog\`ene du tore $G_W$,
 par le lemme \ref{lem:tore}
il y a un isomorphisme canonique
$G_{W*}\otimes_\Z \Zp \isoto \pi_1^\et(W,w)^{(p')}(-1).$
Puisque $G_{W*}=\Hom(\widehat{G_W},\Z)=\Ext_\Z^0(\widehat{G_W},\Z)$, on obtient un isomorphisme canonique
$$
\Ext_\Z^0(\widehat{G_W},\Z)\otimes_\Z \Zp \isoto \pi_1^\et(W,w)^{(p')}(-1),
$$
ce qui d\'emontre le th\'eor\`eme pour $(G_W,W)$.

Comme $W=Z$ et $\Ext_\Z^0(\widehat{G_W},\Z)
=\Ext^0_\Z([\widehat{G_Z}\to\widehat{H_Z}\rangle,\Z)$ (voir \S\,\ref{subsec:case-Z}),
on obtient un isomorphisme canonique
\begin{equation}\label{eq:Z}
\Ext^0_\Z([\widehat{G_Z}\to\widehat{H_Z}\rangle,\Z) \otimes \Z_{(p')} \isoto \pi_1^\et(Z,z)^{(p')}(-1) \, ,
\end{equation}
ce qui d\'emontre le th\'eor\`eme pour $(G_Z,Z)$.

On sait que $\widehat{G_Y} = \widehat{G_Z}$ et $\widehat{H_Y} = \widehat{H_Z}$,
donc on  d\'eduit  de \eqref{eq:YZ} et \eqref{eq:Z} un isomorphisme canonique
$$
\Ext^0_\Z([\widehat{G_Y}\to\widehat{H_Y}\rangle,\Z) \otimes \Z_{(p')} \isoto \pi_1^\et(Y,y)^{(p')}(-1) \, ,
$$
ce qui d\'emontre le th\'eor\`eme pour $(G_Y,Y)$.

\def\foneprime{{ \framebox{\makebox[\totalheight]{$1'$}} }}

D\'eduisons-en maintenant le th\'eor\`eme pour $(G,X)$ : consid\'erons le diagramme suivant \`a lignes exactes:
{\small
\begin{equation}\label{eq:diag-p'}
\xymatrix@C=5mm{
\Ext^0_\Z(\widehat{Q},\Z) \otimes \Z_{(p')} \ar[r] \ar[d]^{\cong} \ar@{}[dr]|{\foneprime}
& \Ext^0_\Z([\widehat{G_Y}\to\widehat{H_Y}\rangle,\Z)
\otimes \Z_{(p')} \ar[d]^{\cong} \ar[r] & \Ext^0_\Z([\widehat{G}\to\widehat{H}\rangle,\Z) \otimes \Z_{(p')} \ar@{-->}[d] \ar[r] & 0 \\
{\pi}_1^\et(Q,1)^{(p')}(-1) \ar[r]^{\lambda_*} &\pi_1^\et(Y,y)^{(p')}(-1)\ar[r]^{\varphi_*} & \pi_1^\et(X,x)^{(p')}(-1)\ar[r] & 0 \, ,
}
\end{equation}
}
dont l'exactitude de la seconde ligne a \'et\'e d\'emontr\'ee plus haut, et dont celle de la premi\`ere
provient de la suite exacte \eqref{eq:exact1} et
de l'exactitude \`a droite du produit tensoriel.
On d\'emontre que le rectangle \foneprime est commutatif comme on d\'emontre
 la commutativit\'e du rectangle \fone
du diagramme \eqref{eq:Q}.
Le diagramme  \eqref{eq:diag-p'} permet bien de d\'efinir la fl\`eche en pointill\'es,
dont on d\'emontre comme en \S\,\ref{subsec:independant}
qu'elle ne d\'epend pas du plongement $j : H^\tor \to Q$. Finalement, cela d\'emontre que l'on a bien un isomorphisme canonique
$$
\Ext^0_\Z([\widehat{G}\to\widehat{H}\rangle,\Z) \otimes \Z_{(p')} \isoto \pi_1^\et(X,x)^{(p')}(-1) \, ,
$$
ce qui conclut la preuve du th\'eor\`eme \ref{thm:charp non connexe-bis}.
\qed
\end{subsec}

\medskip
\noindent
{\bf Remerciements :} Nous remercions  chaleureusement Tam\'as Szamuely pour ses pr\'ecieux commentaires.

\bigskip
{\small

{\scshape
Borovoi: Raymond and Beverly Sackler  School of Mathematical Sciences, Tel Aviv University,
6997801 Tel Aviv, Israel }
\smallskip

{\it E-mail: }
\url{borovoi@post.tau.ac.il}
\bigskip

{\scshape
Demarche: Sorbonne Universit\'es, UPMC Univ Paris 06, IMJ-PRG, UMR 7586 CNRS, Univ Paris Diderot, Sorbonne Paris Cit\'e, F-75005, Paris, France}
\smallskip

{\it E-mail: }
\url{cyril.demarche@imj-prg.fr}
}

\end{document}